\documentclass[12pt,a4paper]{amsart}
\usepackage{setspace}
\usepackage[inner=2.5cm,outer=2.5cm,bottom=2cm]{geometry}
\usepackage{amsmath, amssymb, amsthm, amsfonts}
\usepackage{tikz}
\usetikzlibrary{arrows, babel, knots}
\usepackage{tikz-cd}
\usepackage{xurl}
\usepackage{hyperref}
\usepackage{wrapfig}
\usepackage{multicol}
\usepackage{capt-of}
\usepackage[capitalise,noabbrev]{cleveref}
\definecolor{darkslate}{HTML}{335C67}
\definecolor{vanilla}{HTML}{FFF3B0}
\definecolor{honey}{HTML}{E09F3E}
\definecolor{brownred}{HTML}{9E2A2B}
\hypersetup{colorlinks=true, linkbordercolor=brownred,linkcolor=brownred}
\hypersetup{citecolor=brownred}
\pgfdeclarelayer{background}
\pgfsetlayers{background,main}

\newtheorem{thm}{Theorem}[section]
\newtheorem{cor}[thm]{Corollary}
\newtheorem{lem}[thm]{Lemma}
\newtheorem{prop}[thm]{Proposition}
\theoremstyle{definition}
\newtheorem{defn}[thm]{Definition}
\theoremstyle{remark}
\newtheorem{rmk}[thm]{Remark}

\newtheorem{exm}[thm]{Example}

\newcommand{\CB}{\mathcal{B}}

\newcommand{\CP}{\mathcal{P}}
\newcommand{\CT}{\mathcal{T}}
\newcommand{\CH}{\mathcal{H}}
\newcommand{\CQ}{\mathcal{Q}}

\newcommand{\NC}{{\mathbf C}}
\newcommand{\NR}{{\mathbf R}}
\newcommand{\NQ}{{\mathbf Q}}
\newcommand{\NT}{{\mathbf T}}

\def\NN{\ensuremath{\mathbb{N}}}
\def\QQ{\ensuremath{\mathbb{Q}}}
\def\RR{\ensuremath{\mathbb{R}}}

\def\ZZ{\ensuremath{\mathbb{Z}}}

\newcommand{\ModelCategoryDefinition}[3]{
\begin{description}
\item[$\overset\sim\to$] #1
\item[$\hookrightarrow$] #2
\item[$\twoheadrightarrow$] #3
\end{description}}

\newcommand{\coker}{\operatorname{coker}}
\newcommand{\im}{\operatorname{Im}}

\newcommand{\cyl}{\operatorname{cyl}}
\newcommand{\cocyl}{\operatorname{cocyl}}

\newcommand{\Fun}{\operatorname{Fun}}

\newcommand{\colim}{\operatorname{colim}}
\DeclareMathOperator*{\scolim}{colim}
\newcommand{\hocolim}{\operatorname{hocolim}}
\newcommand{\Ab}{{\operatorname{Ab}}}
\newcommand{\Ch}{\operatorname{Ch}}
\newcommand{\Mod}{\operatorname{-Mod}}
\newcommand{\op}{{\operatorname{op}}}
\newcommand{\Aut}{{\operatorname{Aut}}}
\newcommand{\direct}[1]{\overrightarrow{#1}}
\newcommand{\invers}[1]{\overleftarrow{#1}}
\newcommand{\alt}{\mathrm{height}}
\newcommand{\depth}{\mathrm{depth}}
\renewcommand{\lim}{\operatorname{lim}}
\DeclareMathOperator*{\slim}{lim}
\newcommand{\End}{\operatorname{End}}

\title[]{A homotopical algebra approach to higher limits}
\author[G. Carrión Santiago]{Guille {Carrión~Santiago}}
\address{Departamento de Matemática Aplicada, Ciencia e Ingeniería de los Materiales y Tecnología Electrónica, Universidad Rey Juan Carlos, 28933 Móstoles (Madrid), Spain}
\email{guille.carrions@urjc.es}
\date{2026}
\subjclass[2020]{
Primary 18G10,  
18N40;          
55P99;          
Secondary 06A07,
05C65,          
57K18           
}
\keywords{
Higher limits, fibrant replacement, model category, vanishing bounds, poset, hypergraph cohomology, Khovanov homology, Mackey functors.
}

\usepackage{todonotes}

\begin{document}
\begin{abstract}
We develop a new technique for computing higher limits of functors over filtered posets by constructing explicit fibrant replacements within a suitable model category structure. We apply this procedure to develop two systematic vanishing bounds for the higher limits. For the first one, we define a combinatorial bound derived from the poset's Hasse diagram. This bound systematically outperforms classical poset-length bounds and provides critical stopping criteria for computing the sheaf cohomology of hypergraphs and higher-order interaction networks. The second one is an inductive bound driven by the local algebraic information of the functor. We demonstrate the robustness of this theoretical framework by applying it to two distinct domains: establishing vanishing bounds for Mackey functors over posets with local quasi-units, and bounding the unnormalised Khovanov homology of the torus knot \(T(3,4)\).
\end{abstract}
\maketitle
\section*{Introduction}

Higher limits, the right-derived functors of $\lim$, naturally arise in homotopy theory, homological algebra, and combinatorics, and  admit an interpretation as a generalised cohomology with local coefficients in a functor $F$ over a small category $\CP$,
\[
H^i(\CP;F) =\operatorname{R}^i\lim_\CP(-)(F).
\]
In the realm of homotopy theory, higher limits have proven to be an invaluable tool for studying the cohomology groups of spaces built from pieces via a colimit diagram. In their seminal work \cite{Bousfield1972}, Bousfield and Kan established a spectral sequence whose $E_2$-page relates the cohomology of the homotopy colimit with the higher limits of the cohomology of the objects in the diagram as follows:
\[
H^i(\CP;H^j(F))\Longrightarrow H^{i+j}(\hocolim_\CP F).
\]

Higher limits also control an obstruction theory for the existence and uniqueness of maps out of a homotopy colimit \cite{Wojtkowiak}; see \cite{Jackowski1994} for an overview. There are more recent examples of related problems, such as the obstruction for the existence and uniqueness of the classifying space of a fusion system, as presented in Broto, Levi, and Oliver's work \cite{Broto2003}, or the existence of homotopy representations for compact $p$-local groups, shown by Cantarero and Castellana \cite{CC17}.

The most significant results regarding the $\lim$-acyclicity of functors over orbit categories are linked to Mackey functors \cite{Jackowski1992, Diaz2015} and Lambda functors \cite{JMO1992b, JMO1992a}. The fact that higher limits over an orbit category can be reduced to higher limits over posets \cite{Slominska2001,MR4438061,MR4709409} motivates the development of combinatorial methods to compute higher limits over posets, or at least provide vanishing conditions. Acyclicity conditions in the context of posets, such as the Mittag-Leffler conditions, are widely recognised \cite{Weibel1994}, along with conditions related to projectivity \cite{DiazRamos2009}, lower factoring sections \cite{KishimotoLevi22}, and Mackey functors over posets \cite{CD23}. While these results provide conditions for complete acyclicity, our approach focuses on deriving explicit vanishing bounds.

Beyond pure mathematics, the language of functors over posets has recently proved its worth in applied mathematics. In particular, cellular sheaf-theoretic methods — that is, functors over posets — provide a natural framework for topological data analysis \cite{Curry14} and the cohomological study of networks \cite{Robinson2017, HansenGhrist2021}.

The framework to compute higher limits of functors over posets presented here is based on the notion of fibrant replacement. We set a model category structure in the category of functors from $\CP^\op$ to $\Ch(R)$, the category of chain complexes of $R$-modules, in which higher limits of functors can be computed systematically by means of fibrant replacements.\smallskip

\noindent \textbf{\textcolor{brownred}{\cref{hig_lim_fibrant_rep}.}}\,
\textit{ Let $\CP$ be a filtered poset  and $R$ be a unital commutative ring. Given a functor $F\colon \CP^\op\to R\Mod$, we have
\[
H^*(\CP;F)=H^*(\lim\NR F),
\]
where $\NR F\colon \CP^\op\to \Ch(R)$ is a fibrant replacement of $F$, see Proposition~\ref{thm:model_category}.}\medskip

The inductive structure of $\CP$ reduces the construction of a fibrant replacement of $F\colon \CP^\op\to R\Mod$ to a factorisation problem at each $p\in \CP$ (see Definition~\ref{def_fibrant_rep}):
\[
\begin{tikzcd}
F(p) \arrow[r,"\epsilon"] \arrow[d, "\sim", dashed] & \lim_{\CP_{<p}}F \arrow[d,"\eta"] \\
\NR F(p) \arrow[r, two heads, dashed]    & \lim_{\CP_{<p}}\NR F.  
\end{tikzcd} 
\]
This inductive approach has a clear advantage: we can track the local information to produce vanishing conditions as shown in \cite{CD23} or \cite{CD-shell}.  In this paper, the combinatorial nature of this inductive procedure gives an explicit handle on the height of $\NR F$, which is precisely what controls the vanishing of the higher limits. We present two such bounds: both sharpen the classical one, given by the length of the poset, by tracking more carefully how the height of the fibrant replacement evolves along the inductive construction.

\begin{minipage}{.29\textwidth}
\centering
\begin{tikzpicture}[scale=0.9,
poset_node/.style={circle, draw=darkslate!80, fill=darkslate!10, thick, inner sep=1.5pt, minimum size=6mm},
poset_node_label1/.style={circle, draw=vanilla!90!black, fill=vanilla!30, thick, inner sep=1.5pt, minimum size=6mm},
poset_node_label2/.style={circle, draw=honey!80, fill=honey!10, thick, inner sep=1.5pt, minimum size=6mm},
poset_node_label3/.style={circle, draw=brownred!80, fill=brownred!10, thick, inner sep=1.5pt, minimum size=6mm},
connection/.style={->, >=latex, thick, darkslate!70},x=1.5cm,y=1.2cm
]
\node[poset_node] (A) at (0, 0) {$0$};

\node[poset_node_label1] (B1) at (-1, 1) {$1$};
\node[poset_node_label1] (B2) at (1, 1) {$1$};

\node[poset_node_label1] (C1) at (-1, 2) {$1$};
\node[poset_node_label1] (C2) at (0, 2) {$1$};
\node[poset_node_label1] (C3) at (1, 2) {$1$};

\node[poset_node_label2] (D1) at (-1, 3) {$2$};
\node[poset_node_label1] (D2) at (1, 3) {$1$};

\node[poset_node_label2] (E1) at (-1, 4) {$2$};
\node[poset_node_label1] (E2) at (1, 4) {$1$};

\node[poset_node_label3] (F) at (0, 5) {$3$};

\draw[connection] (A) -- (B1);
\draw[connection] (A) -- (B2);

\draw[connection] (B1) -- (C1);
\draw[connection] (B2) -- (C2);
\draw[connection] (B2) -- (C3);

\draw[connection] (C1) -- (D1);
\draw[connection] (C2) -- (D1);
\draw[connection] (C2) -- (D2);

\draw[connection] (D1) -- (E1);
\draw[connection] (D2) -- (E2);

\draw[connection] (E1) -- (F);
\draw[connection] (E2) -- (F);

\end{tikzpicture}
\captionof{figure}{}
\label{Intro/fig/poset_label}
\end{minipage}\hfill
\begin{minipage}{.59\textwidth}
The first vanishing bound depends only on the geometry of the poset: it replaces the classical bound given by the poset length with a sharper labelling function that tracks the maximum possible height of the fibrant replacement of any functor in each step.

The labelling of a poset $\CP$ is a map of posets ${B\colon \CP\to \NN}$ that essentially counts concrete circuits in the Hasse diagram of the poset, see Definition~\ref{labelling_function}. Figure~\ref{Intro/fig/poset_label} shows a labelled poset. This combinatorial function bounds the vanishing of the higher limits as follows:\smallskip

\noindent\textbf{\textcolor{brownred}{\cref{bound_by_B}.}}\,
\textit{
Let $\CP$ be a filtered poset, and ${B\colon \CP\to \NN}$ be its associated labelling function. Given a functor ${F\colon \CP^\op\to R\Mod}$, we have:
\[
H^k(\CP;F)=0
\]
for every $k>\sup B$.
}
\end{minipage}\medskip

We relate this labelling function to the measure of a weak poset introduced by Mitchell \cite{MR294454}. Mitchell's measure acts as an invariant that quantifies the combinatorial complexity of the poset's intervals, specifically capturing the presence of nested circuits and branching paths. The bound presented by Mitchell is sharper than the one presented here. However, it is not a bound for the height of the fibrant replacement construction.

The bound in \cref{bound_by_B} is particularly effective when applied to hypergraphs arising from network science. A \emph{network} models a system of entities (nodes) and their interactions (edges); when an interaction involves three or more participants simultaneously, the appropriate combinatorial framework is a hypergraph rather than a graph. The topological study of such networks has attracted considerable attention in the applied mathematics community, with sheaf-theoretic methods being employed, for example, in sensor integration \cite{Robinson2017}, opinion dynamics \cite{HansenGhrist2021}, and higher-order neural message-passing \cite{Bodnar2022NeuralSheafDiffusion,HumeLio2025}. In \cref{sect:combinatorial_bound} we
illustrate its effectiveness on two publicly available hypergraph datasets: a co-authorship hypergraph extracted from the \texttt{arXiv} metadata \cite{arxiv_kaggle_dataset}, where $\sup B = 2$ while the poset length equals $3$; and a plant-pollinator interaction hypergraph \cite{lucas_2024_13753744}, where $\sup B = 3$ versus a poset length of $4$. The computational code is available at GitHub \cite{Carrion-Github}.

The second vanishing bound focuses on the \emph{local algebraic information} of $F$ rather than on the geometry of the poset. A classical instance of this principle is the cohomology of a CW-complex: for each cell $p$ of $X$, the boundary $\partial p \cong S^{d(p)-1}$ has the property that $H^k(\partial p;\ZZ)=0$ for $k\ge d(p)$, and this purely local condition forces $H^k(X;\ZZ)=0$ for $k>\dim X$. The following result generalises this classical observation to arbitrary filtered posets and arbitrary functors: the local vanishing of $H^k(\CP_{<p};F)$ for every $p\in\CP$ controls the global vanishing over the entire poset:\smallskip

\textbf{\noindent\textcolor{brownred}{\cref{Thm-inductive1}.}}\,
\textit{Let $\CP$ be a filtered poset and $R$ be a unital commutative ring. Let $F\colon \CP^\op\to R\Mod$ be a functor. If there exists an integer $n\ge 0$ such that for every $p\in \CP$, $H^k(\CP_{< p};F)=0$ for all $k\ge n$, then:
\[
H^k(\CP;F)=0
\]
for all $k>n$.
}\medskip

A weaker version of this result (see \cref{Thm-inductive2}) yields vanishing bounds for Mackey functors with local quasi-units. This result does not require any shellability condition on the poset, unlike \cite[Theorem~5.3]{CD-shell}.\smallskip

\textbf{\noindent\textcolor{brownred}{\cref{thm:VB_inductive_weak_mackey}.}}\quad
\textit{
Let $\CP$ be a filtered poset of length $\ell$ and $R$ be a unital commutative ring. Let $m\le \ell$ be a non-negative integer. If $G\colon \CP^\op\to R\Mod$ is a weak Mackey functor that has a quasi-unit in $\CP_{\le p}$ for every $p\in \CP$ with $d(p)\le m$, then $H^k(\CP;G)=0$ for all $k>\ell-m$.
}
\medskip

We also show how \cref{Thm-inductive1} yields explicit vanishing bounds for the unnormalised Khovanov homology of a link. Everitt and Turner \cite{EverittKHov} showed that the Khovanov homology $\overline{KH}^i(L)$ of a link $L$ with $n$ crossings can be expressed as higher limits of a functor $F_{KH}$ over the Boolean poset $\{0,1\}^n$ of smoothings; thus \cref{Thm-inductive1} applies directly.  We illustrate this on the torus knot $T(3,4)$, also known as $8_{19}$.

Finally, the techniques presented in this paper for bounding higher limits via fibrant replacements dualise effortlessly to higher colimits via cofibrant replacements. By setting up a dual model category structure, we deduce an analogous theory governing the higher colimits.

\textbf{Outline of this paper:} This paper is organised as follows. In \cref{sect preliminars}, we introduce the notation we will use and briefly overview posets and homological algebra. In \cref{sect HoThe}, we establish the model category structure in which higher limits and colimits can be computed by explicit fibrant and cofibrant replacements.  In \cref{sect:combinatorial_bound} we introduce a combinatorial labelling function and prove the corresponding vanishing bound (\cref{bound_by_B}); we also relate it to Mitchell's measure (\cref{thm:delta<=B}) and illustrate its sharpness on two hypergraph datasets. Finally, \cref{sec:Inductive_bound} develops an inductive vanishing bound driven by the local acyclicity of the functor (\cref{Thm-inductive1}), with applications to Mackey functors (\cref{thm:VB_inductive_weak_mackey}) and to the Khovanov homology of the torus knot $T(3,4)$.

\textbf{Acknowledgments:} The author sincerely thanks Natàlia Castellana and Antonio Díaz for their invaluable guidance and useful comments.
The author also thanks O'Bryan Cárdenas-Andaur for his advice on drawing and visualizing knots, and Gonzalo Contreras Aso for fruitful conversations and for introducing the author to the software library used for hypergraph experimentation.

The author was supported by Universidad de Málaga grant G RYC-2010-05663, Comissionat per Universitats i Recerca de la Generalitat de Catalunya (grant No. 2021-SGR-01015), and MICINN grants PID2023-149804NB-I00, PID2020-118753GB-I00 and PID2020-116481GB-I00.
Part of this article is based upon the author's Ph.D. thesis, which was supported by the predoctoral grant BES-2017-079851.

\section{Preliminaries}
\label{sect preliminars}

The two subsections below collect the combinatorial and homological notions used 
throughout the paper: \cref{subsec:posets} establishes the language of filtered posets 
and their subposets, while \cref{subsec:HomAlg} introduces the mapping (co)cylinder 
constructions and their truncated variants, which are the key technical tools 
in the (co)fibrant replacement constructions of \cref{sect HoThe}.

\subsection{Posets}
\label{subsec:posets}
A poset ${\CP=(\CP,\le)}$ is considered a category whose objects are the elements $p\in \CP$ and a single arrow $p\to q$ exists if and only if $p\le q$. We abuse notation and will denote by $\CP$ either the poset as a set or the associated category. A poset $\CP$ is said to be filtered if there exists a (strictly) order-preserving map $\mathrm d\colon \CP\to \NN$. For the purpose of this paper, we can always assume that  $\mathrm d(p)$ equals the maximum of the length of chains that end in $p$,
\[
p_0<p_1<\dots<p_n=p.
\]
Thus, we define the length of $\CP$ to be the maximum of the length of any chain in $\CP$, i.e.,
\[
\operatorname{length}(\CP)=\max \{d(p)\mid p\in \CP\}.
\]
For a poset $\CP$ and an object $p\in \CP$ we denote by $\CP_{< p}$ the subposet  consisting of those $q\in \CP$ such that $q< p$, 
\[
\CP_{< p}=\{q\in \CP\mid q< p\}.
\]
Similarly, we denote by $\CP_{\le p}$, $\CP_{\ge p}$, and $\CP_{>p}$ the evident subposets.
The cover relation on $\CP$, denoted by $\prec$, is the relation on $\CP$ defined by: $p\prec q$ if and only if 
\[
p< q\mbox{ and }\forall r\in \CP, p\le r\le q\Rightarrow p=r\text{ or }q=r.
\]

A subset $Q$ of a poset $\CP$ is \emph{upper convex} if $\CP_{\ge x}\subset Q$ for every $x\in Q$. Dually, $Q$ is \emph{lower convex} if $\CP_{\le x}\subset Q$ for every $x\in Q$.  The \emph{Hasse diagram} of a poset $\CP$ is a directed graph whose vertices are the elements of $\CP$, and there is a directed edge between $x$ and $y\in \CP$ if and only if $x\prec y$. Throughout this manuscript, we draw these edges upward, so the partial order can be read visually from bottom to top.

\subsection{Homological algebra}
\label{subsec:HomAlg}

The technical tool in this paper is the mapping (co)cylinder construction and  its truncated variant, which allow us to factor certain morphisms of cochain complexes in a height-controlled way. Let $R$ be a unital commutative ring. We denote by $\Ch(R)$ the category of (unbounded) cochain complexes. For practical reasons, given a cochain complex $C$, we use the convention that $C_k=C^{-k}$.
\begin{defn}
\label{mapping_co_cyl}
Let $f\colon C\to D$ be a map between cochain complexes.
\begin{enumerate}
\item The \emph{mapping cocylinder} of $f$ is the cochain complex $\cocyl(f)$ defined by
\[
\cocyl(f)^k=C^k\times D^{k-1}\times D^k,
\]
in degree $k$ and the differential is given by the formula
\[
\partial(c,d,d')=(\partial_C c,d'-f(c)-\partial d,\partial_D d').
\]
\item The \emph{mapping cylinder} of $f$ is the cochain complex $\cyl(f)$ defined by
\[
\cyl(f)_k=C_k\times C_{k-1}\times D_k,
\]
in degree $k$ and the differential is given by the formula
\[
\partial(c,c',d)=(\partial_C c+c',-\partial_C c',\partial_D d-f(c')).
\]
\end{enumerate}
\end{defn}

\begin{rmk}
\label{fact_cyl_cocyl}
Given a map between cochain complexes $f\colon C\to D$, the mapping cylinder provides a factorisation of $f$ as a monomorphism followed by a quasi-isomorphism:
\begin{equation}
\label{eq:fact_cyl}
\begin{tikzcd}[row sep=tiny]
C\arrow[r,"i"]&\cyl(f)\arrow[r,"\pi"]& D\\
c\arrow[r,maps to] &(c,0,0)&\\
&(c,c',d)\arrow[r,maps to] &f(c)+d,
\end{tikzcd}
\end{equation}
and the mapping cocylinder factors $f$ as a quasi-isomorphism followed by an epimorphism:
\begin{equation}
\label{eq:fact_cocyl}
\begin{tikzcd}[row sep=tiny]
C\arrow[r,"i"]&\cocyl(f)\arrow[r,"\pi"]& D\\
c\arrow[r,maps to] &(c,0,f(c))&\\
&(c,d,d')\arrow[r,maps to] &d'.
\end{tikzcd}
\end{equation}
For more details see \cite[Section 1.5]{Weibel1994}.
\end{rmk}

\begin{defn}
Let $C$ be a bounded cochain complex and ${F\colon \CP\to \Ch(R)}$ be a functor from a small category $\CP$.
\begin{enumerate}
\item We define the \emph{height} of $C$, denoted by $\alt(C)$, to be the integer $n$ such that $C^k=0$ for all $k>n$ and $C^n\neq 0$. Similarly, we define the \emph{height} of $F$, which we will also denote $\alt(F)$, to be the supremum of the heights, i.e.,
\[
\alt(F)=\sup\{\alt(F(p))\mid p\in \CP\}.
\]
\item We define the \emph{depth} of $C$, denoted by $\depth(C)$, to be the integer $n$ such that $C_k=0$ for all $k>n$ and $C_n\neq 0$. Similarly, we define the \emph{depth} of $F$, which we will also denote $\depth(F)$, to be the supremum of the depths, i.e.,
\[
\depth(F)=\sup\{\depth(F(p))\mid p\in \CP\}.
\]
\end{enumerate}
\end{defn}
\begin{prop}
\label{rmk_cyl+1}
Let $f\colon C\to D$ be a morphism between cochain complexes. Assume that ${\alt(C)<\alt(D)}$ (resp. $\depth(C)>\depth(D)$). Then
\[
\alt(\cocyl(f))=\alt(D)+1\qquad \text{(resp. } \depth(\cyl(f))=\depth(C)+1\text{)}.\hfill\qed
\]
\end{prop}
\begin{defn}
Let $f\colon C\to D$ be a morphism between cochain complexes. The morphism $f\colon C\to D$ is said to be \emph{truncatable} if $n=\depth(C)>\depth(D)$, and the differential 
\[
\partial_C\colon {C_{n}\hookrightarrow C_{n-1}}
\]
is a monomorphism.
Dually, the morphism $f\colon C\to D$ is said to be \emph{cotruncatable} if $\alt(C)<\alt(D)=m$, and the differential 
\[
\partial_D\colon {D^{m-1}\twoheadrightarrow D^{m}}
\]
is an epimorphism.
\end{defn}

When $f$ is (co)truncatable, the following variant of the mapping (co)cylinder  provides a factorisation analogous to Remark~\ref{fact_cyl_cocyl} without increasing the depth (height) of the resulting complex.

Let $\NT_n\colon \Ch(R)\to \Ch(R)$ and $\NC^n\colon \Ch(R)\to \Ch(R)$ denote the $n$-truncation and $n$-cotruncation functor defined on a cochain complex $C$ by:

\begin{centering}
\begin{minipage}{0.6\linewidth}
\[
(\NT_n C)_k=\left\lbrace \begin{matrix}
C_k&\mbox{if } k\le n\\
0&\mbox{if } k>n,
\end{matrix}\right.
\]
\end{minipage}
\begin{minipage}{0.3\linewidth}
\[
(\NC^n C)^k=\left\lbrace \begin{matrix}
C^k&\mbox{if } k\le n\\
0&\mbox{if } k>n,
\end{matrix}\right.
\]
\end{minipage}
\end{centering}

\noindent with the obvious action on morphisms. In the case of truncating at $\depth(C)$ (resp. $\alt(C)$), we omit the index and write $\NT C$ (resp. $\NC C$).

\begin{defn}
\label{def:truncated_mapping_cocyl}
Let $f\colon C\to D$ be a morphism between cochain complexes.
\begin{enumerate}
\item If $f$ is cotruncatable we define the \emph{truncated mapping cocylinder} of $f$, denoted by $\cocyl_{\NT}(f)$, as
\[
\cocyl_{\NT}(f)=\cocyl\Big(C\overset{\NC(f)}\longrightarrow \NC D\Big).
\]
\item If $f$ is truncatable we define the \emph{truncated mapping cylinder} of $f$, denoted by $\cyl_{\NT}(f)$, as
\[
\cyl_{\NT}(f)=\cyl\Big(\NT C\overset{\NT(f)}\longrightarrow  D\Big).
\]
\end{enumerate}
\end{defn}

\begin{prop}
\label{fact_truncatedcocyl}
Let $f\colon C\to D$ be a morphism between cochain complexes. 
\begin{enumerate}
\item If $f$ is truncatable, then it factors through the truncated mapping cylinder as a monomorphism followed by a quasi-isomorphism,
\[
\begin{tikzcd}
C \arrow[r,hook] & \cyl_\NT(f) \arrow[r, "\sim"] & D.
\end{tikzcd}
\]
\item If $f$ is cotruncatable, then it factors through the truncated mapping cocylinder as a quasi-isomorphism followed by an epimorphism,
\[
\begin{tikzcd}
C \arrow[r, "\sim"] & \cocyl_{\NT}(f) \arrow[r, two heads] & D.
\end{tikzcd}
\]
\end{enumerate} 
\end{prop}
\begin{proof}
We only prove the cotruncatable case; the truncated case is similar. The inclusion $i\colon C\to \cocyl_{\NT}(f)$ is the one given in \eqref{eq:fact_cocyl}. The morphism {$\pi\colon \cocyl_\NT(f)\twoheadrightarrow D$} is given by:
\[
\pi^k=\left\lbrace \begin{matrix}
0 & \mbox{if } k>n+1\\
\partial_D:D^{n}\to D^{n+1}&\mbox{if }k=n+1\\
\pi_{D^k}:C^k\times D^{k-1}\times D^k\to  D^k &\mbox{if } k\le n,
\end{matrix}\right. 
\]
where $\pi_{D^k}$ is the projection and $\partial_D$ is the differential of $D$.
\end{proof}

\section{Homotopical algebra approach to higher limits}
\label{sect HoThe}
Given a filtered poset $\CP$ and a unital commutative ring $R$, the limit and colimit define left-exact and right-exact functors, respectively:
\[
\lim\colon \Fun(\CP^\op,R\Mod)\to R\Mod,\qquad
\colim\colon \Fun(\CP,R\Mod)\to R\Mod.
\]
The \emph{higher limits} of a contravariant functor $F\colon\CP^\op\to R\Mod$ are the right-derived functors of the limit,
\[
H^k(\CP;F):=\operatorname{R}^k\lim(F),
\]
and the \emph{higher colimits} of a covariant functor $F\colon\CP\to R\Mod$ are the left-derived functors of the colimit,
\[
H_k(\CP;F):=\operatorname{L}_k\colim(F).
\]

This section presents a uniform approach to the computation of both invariants by setting a model category structure in the category of functors. For background on model categories see \cite{Balchin2021, Dwyer1995} or Quillen's original work \cite{Quillen1967}. Throughout this paper, every functor $F\colon\CP^\op\to R\Mod$ (resp.\ $F\colon\CP\to R\Mod$) is regarded as an object of $\Fun(\CP^\op,\Ch(R))$ (resp.\ $\Fun(\CP,\Ch(R))$), with $F(p)$ concentrated in degree $0$ for every $p\in\CP$. We develop the theory in detail for higher limits and contravariant functors; each result is followed immediately by its dual statement for higher colimits and covariant functors, stated without proof since the arguments are entirely symmetric.

\begin{prop}
\label{thm:model_category}
Let $\CP$ be a filtered poset, and $R$ be a unital commutative ring. There is a model category structure on the category of functors $\Fun(\CP^\op,\Ch(R))$ in which a natural transformation $\eta\colon F\Rightarrow G$ is a 
\ModelCategoryDefinition%
{weak equivalence if, for every object $p\in \CP$, the morphism 
$
\eta_p\colon F(p)\to G(p)
$
is a quasi-isomorphism, that is, it induces an isomorphism in cohomology;}%
{cofibration if, for every object $p\in \CP$, the morphism 
$
\eta_p\colon F(p)\to  G(p)
$
is a split monomorphism with cofibrant cokernel; and}%
{fibration if, for every object $p\in \CP$, the morphism
\[
F(p)\to G(p)\times_{{\displaystyle\slim_{\CP_{<p}}G}}\slim_{\CP_{<p}}F,
\]
is an epimorphism.
}%
\end{prop}
\begin{proof}
We consider $\CP$ as a Reedy category, see \cite[Section 15]{Hirschhorn2003} for background, by setting $\direct\CP=\CP$ and $\invers{\CP}$ equal to the discrete category spanned by the objects of $\CP$. We endow $\Ch(R)$ with the projective model category structure, see \cite[Subsection 18.4]{May2012}. Then, the proof follows by applying the result of D. M. Kan \cite[Theorem 15.3.4]{Hirschhorn2003}.
\end{proof}
\begin{cor}
\label{functor_fibrante}
Let $\CP$ be a filtered poset, and $R$ be a unital commutative ring. A functor ${F\colon \CP^\op\to \Ch(R)}$ is fibrant if and only if, for every $p\in \CP$, the natural map
\begin{equation}
\label{condicion_fibrante}
F(p)\to \slim_{\CP_{<p}}F
\end{equation}
is an epimorphism.\qed
\end{cor}
\begin{prop}
\label{thm:dual_model_category}
Let $\CP$ be a filtered poset, and $R$ be a unital commutative ring. Then there is a model category structure on the category of functors $\Fun(\CP, \Ch(R))$ in which a natural transformation $\eta\colon F\Rightarrow G$ is a
\ModelCategoryDefinition%
{weak equivalence if, for every object $p\in \CP$, the morphism $\eta_p\colon F(p)\to G(p)$ induces an isomorphism in homology;}%
{cofibration if, for every object $p\in \CP$, the morphism 
\[
F(p)\bigsqcup_{\displaystyle\scolim_{\CP_{<p}}F}\scolim_{\CP_{<p}}G\to  G(p),
\]
is a monomorphism; and}%
{fibration if, for every object $p\in \CP$, the  morphism $F(p)\to G(p)$ is a split epimorphism with fibrant kernel.
}%
\end{prop}
\begin{cor}
Let $\CP$ be a filtered poset, and $R$ be a unital commutative ring. A functor ${F\colon \CP\to \Ch(R)}$ is cofibrant if and only if, for every $p\in \CP$, the natural map
\begin{equation}
\label{condicion_cofibrante}
\scolim_{\CP_{<p}}F\to F(p)
\end{equation}
is a monomorphism.\qed
\end{cor}

\begin{exm}
\label{ex:poset_initial_object}
Let $\CP$ be a filtered poset with an initial object, and $M$ be an $R$-module. Then, the constant contravariant functor ${\underline M\colon \CP^{\op}\to R\Mod}$ is fibrant. Dually, the constant covariant functor ${\underline M\colon \CP\to R\Mod}$ is cofibrant.
\end{exm}

\begin{exm}
\label{ex:mackey}
Let $K$ be a simplicial set and $\CP$ denote its face poset. Let 
\[
(I,D)\colon \CP(K)\to R\Mod
\]
be a pair of twin functors. Then $I$ is cofibrant and $D$ is fibrant; see  \cite[Definition 2.6]{NotbohmRay_DJ}.
More generally, a contravariant (resp. covariant) weak Mackey functor with quasi-unit is fibrant (resp. cofibrant); see \cref{sec:weak_mackey_functors} for definitions and \cite{CD23} for background.
\end{exm}

\begin{thm}
\label{hig_lim_fibrant_rep}
Let $\CP$ be a filtered poset and $R$ be a unital commutative ring. 
\begin{enumerate}
\item Given a functor $F\colon \CP^\op\to R\Mod$, we have
\[
H^*(\CP;F)=H^*(\lim\NR F),
\]
where $\NR F\colon \CP^\op\to \Ch(R)$ is a fibrant replacement of $F$ in the model category structure defined in Proposition~\ref{thm:model_category}.
\item Given a functor $F\colon \CP\to R\Mod$, we have
\[
H_*(\CP;F)=H_*(\colim\NQ F),
\]
where $\NQ F\colon \CP\to \Ch(R)$ is a cofibrant replacement of $F$ in the model category structure defined in Proposition~\ref{thm:dual_model_category}.
\end{enumerate}
\end{thm}
\begin{proof}
We only prove (1); the proof of (2) is analogous. Weak equivalences and cofibrations in the category of functors $\Fun(\CP^\op,\Ch(R))$ are object-wise weak equivalences and cofibrations in the category of cochain complexes $\Ch(R)$. Therefore, the adjoint pair $\Delta \colon \Ch(R)\Longleftrightarrow \Fun(\CP^\op,\Ch(R))\colon \lim$, where $\Delta$ is the diagonal functor, is a Quillen pair, which implies the theorem (see \cite[Section 9]{Dwyer1995}).
\end{proof}

A direct consequence of this approach is that we can compute vanishing bounds for the higher limits directly from the height of any fibrant replacement.

\begin{cor}
\label{bound_from_height}
Let $\CP$ be a filtered poset, and $R$ be a unital commutative ring.
\begin{enumerate}
\item Let ${F\colon \CP^\op\to R\Mod}$ be a functor. If $\NR F\colon \CP^\op\to \Ch(R)$ is a fibrant replacement of $F$ such that $\alt(\NR F)=n$, then
\[
H^k(\CP;F)=0
\]
for every $k>n$. In particular, if $F$ is fibrant, then $F$ is $\lim$-acyclic, that is, ${H^k(\CP;F)=0}$ for $k>0$.
\item Let ${F\colon \CP\to R\Mod}$ be a functor. If $\NQ F\colon \CP\to \Ch(R)$ is a cofibrant replacement of $F$ such that $\depth(\NQ F)=n$, then
\[
H_k(\CP;F)=0
\]
for every $k>n$. In particular, if $F$ is cofibrant, then $F$ is $\colim$-acyclic, that is, ${H_k(\CP;F)=0}$ for $k>0$.\qed
\end{enumerate}
\end{cor}

\begin{exm}
The functors described in Example~\ref{ex:poset_initial_object} and Example~\ref{ex:mackey} are $\lim$-acyclic (resp.\ $\colim$-acyclic).
\end{exm}

\subsection{Fibrant replacement construction}
The idea underlying the preceding results is that higher limits can be computed through a fibrant replacement rather than an injective resolution. Thus, we now devote this subsection to describing a systematic construction of a fibrant replacement for a given functor. 

The Reedy structure on $\CP$ allows us to follow an inductive strategy to construct a fibrant replacement of a functor $F\colon \CP^\op\to R\Mod$; see \cite[Section 15.2]{Hirschhorn2003}.
We start by defining
\[
\NR F(p)=F(p)
\] 
for every object $p$ of degree $0$. Next, assume that $\NR F$ is already defined in the full subcategory of objects of degree less than $n$. To define $\NR F$ on $p\in \CP$ of degree $n$, we need to choose a factorisation of the composite $\varepsilon_p\colon F(p)\to \lim_{\CP_{<p}}F\to \lim_{\CP_{<p}}\NR F$ as a quasi-isomorphism followed by an epimorphism. In other words, we must choose an object $\NR F(p)$ together with a quasi-isomorphism $F(p)\to \NR F(p)$ and an epimorphism $\NR F(p)\to \lim_{\CP_{<p}}\NR F$ such that the following diagram commutes:
\begin{equation}\label{fibrant_diagram}
\begin{tikzcd}
F(p) \arrow[r] \arrow[d, "\sim", dashed] & \lim_{\CP_{<p}}F \arrow[d] \\
\NR F(p) \arrow[r, two heads, dashed]    & \lim_{\CP_{<p}}\NR F.  
\end{tikzcd} 
\end{equation}

Our desired fibrant replacement should coincide with the original functor on the objects that satisfy the fibrantness condition \eqref{condicion_fibrante} given in the following definition.

\begin{defn}
Let $\CP$ be a filtered poset, and $R$ be a unital commutative ring. A functor $F\colon\CP^\op\to R\Mod$ is said to be \emph{locally fibrant} at $p\in \CP$ if, for every $q\in \CP$ with $q\le p$, the matching morphism:
\[
F(q)\to \slim_{\CP_<q}F
\]
is an epimorphism.
\end{defn}
\begin{rmk}
If $Q$ is a lower convex subposet of $\CP$ and $F\colon \CP^\op\to R\Mod$ is a functor such that $F$ is locally fibrant at $q$ for every $q\in Q$, then $F\vert_Q$ is fibrant; therefore, $H^k(Q;F\vert_Q)=0$ for $k>0$.
\end{rmk}

Now, we define a fibrant replacement of a given functor by providing the factorisation described in \eqref{fibrant_diagram}.

\begin{defn}
\label{def_fibrant_rep}
Let $\CP$ be a filtered poset and $R$ be a unital commutative ring. Let \mbox{$F\colon \CP^\op\to R\Mod$} be a functor. We denote by $\NR F\colon \CP^\op\to \Ch(R)$ the functor defined by the following factorisation of the composite $\varepsilon_p\colon F(p)\to \lim_{\CP_{<p}}\NR F$:
\begin{enumerate}
\item if $F$ is locally fibrant at $p$, $\NR F(p)=F(p)$ with the trivial factorisation;
\item if $\varepsilon_p$ is cotruncatable, $\NR F(p)=\cocyl_{\NT}(\varepsilon_p)$ and the factorisation is provided by Proposition~\ref{fact_truncatedcocyl}; and
\item otherwise, $\NR F(p)=\cocyl(\varepsilon_p)$ with the natural factorisation of the mapping cocylinder, see \eqref{eq:fact_cocyl}.
\end{enumerate}
\end{defn}
\begin{rmk}
\label{alt_pos_no_loc_inj}
Notice that the cotruncatability of the morphism $\varepsilon_p\colon F(p)\to \lim_{\CP_{<p}}\NR F$ implies that $\alt(\lim_{\CP_{<p}}\NR F)>0$. Thus, if $\varepsilon_p$ is cotruncatable, then $F$ cannot be locally fibrant at $p$.
\end{rmk}
\begin{prop}
Let $\CP$ be a filtered poset, $R$ be a unital commutative ring and \mbox{$F\colon \CP^\op\to R\Mod$} be a functor. The functor $\NR F\colon \CP^\op\to \Ch(R)$ is a fibrant replacement for $F$ in the model category structure presented in Proposition~\ref{thm:model_category}.
\end{prop}
\begin{proof}
By construction, the natural transformation $F\Rightarrow \NR F$ is a weak equivalence in the model category structure presented in Proposition~\ref{thm:model_category}, and the fact that $\NR F$ is fibrant is deduced directly by \eqref{eq:fact_cocyl} and Proposition~\ref{fact_truncatedcocyl}.
\end{proof}
\begin{rmk}
\label{rmk:fib-choice-cocyl}
The fibrant replacement $\NR F$ from Definition~\ref{def_fibrant_rep} systematically uses the truncated mapping cocylinder at every object $p$ such that the respective morphism $\varepsilon_p$ is cotruncatable. However, the construction admits mixed variants: one may choose, at any specific object $p\in\CP$, to use the full mapping cocylinder $\cocyl(\varepsilon_p)$ instead, even when $\varepsilon_p$ is cotruncatable. Such a mixed choice still yields a valid fibrant replacement of $F$, but the height at those objects where the full cocylinder is used may be larger, potentially weakening the resulting vanishing bound. Nevertheless, there are structured situations in which deliberately using the full mapping cocylinder at certain objects is already optimal, for an example see \cite{CD-shell}.
\end{rmk}
\begin{defn}
\label{F-trunct}
Let $F\colon \CP^\op\to R\Mod$ be a functor, and ${\NR F\colon \CP^\op\to \Ch(R)}$ be its fibrant replacement as in Definition~\ref{def_fibrant_rep}. We say that $p\in \CP$ is \emph{$F$-truncatable} if the composite
\[
F(p)\to \slim_{\CP_{<p}} F\to \slim_{\CP_{<p}} \NR F
\]
is cotruncatable, i.e., $n=\alt(\slim_{\CP_{<p}} \NR F)>0$ and the last non-zero differential
\[
{\slim_{\CP_{<p}} \NR F^{n-1}\to \slim_{\CP_{<p}} \NR F^{n}}
\]
is an epimorphism.
\end{defn}
\begin{rmk}\label{rmk:alt_NRF}
Given a filtered poset $\CP$ and a functor \mbox{$F\colon \CP^\op\to R\Mod$}, by Remark~\ref{rmk_cyl+1}, we have that 
\[
\alt(\NR F)\le \max\{d(p)\mid p\in \CP\}.
\]
\end{rmk}

We easily deduce the classical vanishing bound for the higher limits given by the length of the maximal chain in $\CP$.

\begin{cor}
\label{vanishingbound_deep}
Let $\CP$ be a filtered poset, $R$ be a unital commutative ring.
Then, for any functor $F\colon \CP^\op\to R\Mod$, we have that
\[H^k(\CP;F)=0\]
for every $k>\operatorname{length}(\CP)$.
\end{cor}
\begin{proof}
This holds by Remark~\ref{rmk:alt_NRF} and Corollary~\ref{bound_from_height}.
\end{proof}
We conclude this section with the next example.
\begin{exm}
\label{example_N}
Let $(\NN,\le)$ be the poset of natural numbers with the trivial filtration ${\mathrm{id}_\NN\colon \NN\to \NN}$, and $F\colon \NN^\op\to R\Mod$ be a functor.  Assume that ${F(0<1)}$ is not an epimorphism. Then, $\NR F$ is defined on objects by: 
\[\NR F(0)=F(0);\qquad \NR F(1)= \cocyl(F(0<1));\]
and, for $n>1$
\[
\NR F(n)^k=\begin{cases}
0&\mbox{if }k\neq 0,1;\\
\prod_{i<n}F(i) &\mbox{if }k=1;\\
\prod_{i\le n} F(i)& \mbox{if }k=0.
\end{cases}
\]
The differential is given by the formula:
\[
\partial(x_0,x_1,\dots,x_n)_i=\left(x_i-\sum_{i<n}F(i<n)(x_n)\right).
\]
Then, by \cref{hig_lim_fibrant_rep}, we have
\[
H^k(\NN; F)\cong\begin{cases}
0&\mbox{if }k>1,\\
\coker \overline\partial&\mbox{if }k=1,\\
\ker \overline\partial\cong \lim_\NN F&\mbox{if }k=0;
\end{cases}
\]
where $\overline \partial \colon \lim \NR F^0\to \lim \NR F^1$ is induced by the object-wise differential.
\end{exm}

\subsection{Cofibrant replacement construction}
Here, we present dual versions of the constructions in the previous section. We do not include proofs since the arguments are identical to the fibrant replacement.

As before, we follow an inductive strategy, and we define 
\[
\NQ F(p)=F(p)
\] 
for every object $p$ of degree $0$. The inductive construction consists of factorising, for every $p\in \CP$ of positive degree, the composite morphism: 
\[
\scolim_{\CP_{<p}}\NQ F\to \scolim_{\CP_{<p}} F\to F(p)
\]
as a monomorphism followed by a quasi-isomorphism, as shown in the following diagram
\begin{equation}\label{cofibrant_diagram}
\begin{tikzcd}
\colim_{\CP_{<p}}\NQ F \arrow[d] \arrow[r, dashed, hook] & \NQ F(p) \arrow[d, "\sim", dashed] \\
\colim_{\CP_{<p}} F \arrow[r]                            & F(p).                      
\end{tikzcd}
\end{equation}
We will construct the cofibrant replacement following the same structure as in the fibrant replacement but using the dual notions.

\begin{defn}
Let $\CP$ be a filtered poset, and $R$ be a unital commutative ring. A functor $F\colon\CP\to R\Mod$ is said to be \emph{locally cofibrant} at $p\in \CP$ if, for every $q\in \CP$ such that $q\le p$, the morphism:
\[
\scolim_{\CP_<q} F\to F(q)
\]
is a monomorphism. 
\end{defn}
\begin{rmk}
Notice that if $Q$ is a lower convex subposet of $\CP$ and ${F\colon \CP\to R\Mod}$ is a functor such that $F$ is locally cofibrant at $q$ for every $q\in Q$, then $F\vert_Q$ is cofibrant; therefore, $H_k(Q;F\vert_Q)=0$ for $k>0$.
\end{rmk}

As before, we define a cofibrant replacement of a given functor by providing the factorisation described in \eqref{cofibrant_diagram}.

\begin{defn}
\label{def_cofibrant_rep}
Let $\CP$ be a filtered poset and $R$ be a unital commutative ring. Let \mbox{$F\colon \CP\to R\Mod$} be a functor. We denote by $\NQ F\colon \CP\to \Ch(R)$ the functor defined by the following factorisation of the composite $\varepsilon_p\colon \colim_{\CP_{<p}}\NQ F\to F(p)$:
\begin{enumerate}
\item if $F$ is locally cofibrant at $p$, $\NQ F(p)=F(p)$ with the trivial factorisation;
\item if $\varepsilon_p$ is truncatable, $\NQ F(p)=\cyl_\NT(\varepsilon_p)$ and the factorisation is provided by Proposition~\ref{fact_truncatedcocyl}; and
\item otherwise, $\NQ F(p)=\cyl(\varepsilon_p)$ with the natural factorisation of the mapping cylinder, see \eqref{eq:fact_cyl}.
\end{enumerate}
\end{defn}
\begin{prop}
Let $\CP$ be a filtered poset, $R$ be a unital commutative ring and \mbox{$F\colon \CP\to R\Mod$} be a functor. The functor $\NQ F\colon \CP\to \Ch(R)$ is a cofibrant replacement for $F$ in the model category structure presented in Proposition~\ref{thm:dual_model_category}.\qed
\end{prop}
\begin{defn}
Let $F\colon \CP\to R\Mod$ be a functor, and ${\NQ F\colon \CP\to \Ch(R)}$ be its cofibrant replacement as in Definition~\ref{def_cofibrant_rep}. We say that $p\in \CP$ is \emph{$F$-truncatable} if the composite
\[
\scolim_{\CP_{<p}}\NQ F \to \scolim_{\CP_{<p}}F\to F(p)
\]
is truncatable, i.e., $n=\depth(\colim_{\CP_{<p}} \NQ F)>0$ and the last non-zero differential
\[
{\scolim_{\CP_{<p}} \NQ F^{n}\to \scolim_{\CP_{<p}} \NQ F^{n-1}}
\]
is a monomorphism.
\end{defn}
\begin{rmk}
Given a filtered poset $\CP$ and a functor \mbox{$F\colon \CP\to R\Mod$}, by Remark~\ref{rmk_cyl+1}, we have that 
\[
\depth(\NQ F)\le \max\{d(p)\mid p\in \CP\}.
\]
\end{rmk}
\begin{cor}
Let $\CP$ be a filtered poset and $R$ be a unital commutative ring. For any functor $F\colon \CP\to R\Mod$, we have that $H_k(\CP;F)=0$ for every $k>\operatorname{length}(\CP)$.
\end{cor}

\section{Combinatorial bound for the vanishing of the higher limits}
\label{sect:combinatorial_bound}
As Example~\ref{example_N} illustrates, the geometry and combinatorics of the poset can significantly influence the vanishing for higher (co)limits — just as the directedness of a poset governs vanishing for colimits; see \cite[Theorem~2.6.15]{Weibel1994}. In particular, the bound provided by Corollary~\ref{vanishingbound_deep} is far from being sharp in general, especially for combinatorially rich posets.
In this section, we introduce a combinatorial labelling function on $\CP$ that controls the height of the (co)fibrant replacement and thereby induces a vanishing bound for the higher (co)limits.

\begin{defn}
Let $\CP$ be a filtered poset, $p\in \CP$ and $S$ be a subposet of $\CP_{<p}$. 
We say that $p$ \emph{closes a circuit in $S$} if a connected component of $S$ has at least two maximal objects.
\end{defn}
From a combinatorial point of view, if ``$p$ closes a circuit in $S$'', then the Hasse diagram of $S\cup\{p\}$ contains an undirected circuit passing through $p$.

\begin{figure}[h!]
\centering
\begin{tikzpicture}[
dot/.style={circle, draw=darkslate, fill=darkslate!20, inner sep=2pt, thick},
connection/.style={draw=darkslate, thin},
regiona/.style={draw=honey, fill=honey!20, rounded corners=10pt, line width=2pt, opacity=0.8},
regionb/.style={draw=brownred, fill=brownred!20, rounded corners=10pt, line width=2pt, opacity=0.8}
]

\begin{scope}[shift={(0,0)}]
\filldraw[regiona] (-0.2, 1.7) rectangle (2.2, 3.3);
\node[text=honey, anchor=north west] at (-0.2, 3.3) {$S$};

\node[dot] (p1) at (1, 4) {}; \node[anchor=south] at (p1.north) {$p$};
\node[dot] (a1) at (0.5, 3) {}; \node[dot] (b1) at (1.5, 3) {};
\node[dot] (c1) at (0, 2) {};   \node[dot] (d1) at (1, 2) {};   \node[dot] (e1) at (2, 2) {};
\node[dot] (f1) at (0.5, 1) {}; \node[dot] (g1) at (1.5, 1) {};
\path[connection] (p1) edge (a1) edge (b1);
\path[connection] (a1) edge (c1) edge (d1);
\path[connection] (b1) edge (d1) edge (e1);
\path[connection] (c1) edge (f1);
\path[connection] (d1) edge (f1) edge (g1);
\path[connection] (e1) edge (g1);
\end{scope}

\begin{scope}[shift={(4,0)}]
\filldraw[regionb] (0.2, 0.8) rectangle (1.8, 3.5);
\node[text=brownred, anchor=north west] at (0.2, 3.5) {$S'$};

\node[dot] (p2) at (1.5, 4) {}; \node[anchor=south] at (p2.north) {$p$};
\node[dot] (a2) at (1, 3) {};   \node[dot] (b2) at (2, 3) {};
\node[dot] (c2) at (0.5, 2) {}; \node[dot] (d2) at (1.5, 2) {}; \node[dot] (e2) at (2.5, 2) {};
\node[dot] (f2) at (1, 1) {};   \node[dot] (g2) at (2, 1) {};
\path[connection] (p2) edge (a2) edge (b2);
\path[connection] (a2) edge (c2) edge (d2);
\path[connection] (b2) edge (d2) edge (e2);
\path[connection] (c2) edge (f2);
\path[connection] (d2) edge (f2) edge (g2);
\path[connection] (e2) edge (g2);
\end{scope}
\end{tikzpicture}
\caption{$p$ closes a circuit in $S$ but not in $S'$.
}
\end{figure}
\begin{defn}
\label{labelling_function}
Let $\CP$ be a filtered poset. We define the \emph{labelling function} of $\CP$ to be the map $B\colon \CP\to \NN$ defined inductively as follows: we assign the value $0$ to objects of degree $0$, and $1$ to objects of degree $1$.
Given $p\in \CP$ of degree $n$, let
\[
m=\max\{B(q)\mid q< p\},\quad\mbox{and}\quad 
\CB_p=\{s\in \CP_{<p}\mid m-1\le B(s)\le m\}.
\]
We define the label at $p$ by the following rule:
\[B(p)=\begin{cases}
m+1 &\mbox{if }p\mbox{ closes a circuit in }\CB_p,\\
m&\mbox{otherwise.}    
\end{cases}
\]
\end{defn}

\begin{figure}[h!]
\centering
\begin{tikzpicture}[
poset_node/.style={circle, draw=honey, fill=honey!50, ultra thick, inner sep=2pt, font=\footnotesize\bfseries, minimum size=6mm},
poset_node_label/.style={circle, draw=brownred, fill=brownred!20, ultra thick, inner sep=2pt, font=\footnotesize\bfseries, minimum size=6mm},
connection/.style={draw=darkslate!60, thick}
]

\begin{scope}
\node[poset_node] (p0) at (0, 0) {$p_0$};
\node[poset_node] (p00) at (2, 0) {$p_1$};

\node[poset_node] (p3) at (0, 1.2) {$p_4$};
\node[poset_node] (p2) at (1, 1.2) {$p_3$};
\node[poset_node] (p1) at (2, 1.2) {$p_2$};

\node[poset_node] (p5) at (1, 2.4) {$p_6$};
\node[poset_node] (p4) at (2, 2.4) {$p_5$};

\node[poset_node] (p6) at (0, 3.6) {$p_7$};
\node[poset_node] (p7) at (2, 3.6) {$p_8$};

\node[poset_node] (p8) at (1, 4.8) {$p_9$};

\draw[connection] (p0) -- (p3);
\draw[connection] (p0) -- (p2);
\draw[connection] (p0) -- (p1);
\draw[connection] (p00) -- (p3);
\draw[connection] (p00) -- (p2);
\draw[connection] (p00) -- (p1);
\draw[connection] (p3) -- (p6);
\draw[connection] (p2) -- (p5);
\draw[connection] (p1) -- (p4);
\draw[connection] (p5) -- (p7);
\draw[connection] (p4) -- (p7);
\draw[connection] (p6) -- (p8);
\draw[connection] (p7) -- (p8);
\end{scope}

\begin{scope}[xshift=5cm]
\node[poset_node_label] (l0) at (0, 0) {$0$};
\node[poset_node_label] (l00) at (2, 0) {$0$};

\node[poset_node_label] (l3) at (0, 1.2) {$1$};
\node[poset_node_label] (l2) at (1, 1.2) {$1$};
\node[poset_node_label] (l1) at (2, 1.2) {$1$};

\node[poset_node_label] (l5) at (1, 2.4) {$1$};
\node[poset_node_label] (l4) at (2, 2.4) {$1$};

\node[poset_node_label] (l6) at (0, 3.6) {$1$};
\node[poset_node_label] (l7) at (2, 3.6) {$2$};

\node[poset_node_label] (l8) at (1, 4.8) {$2$};

\draw[connection] (l0) -- (l3);
\draw[connection] (l0) -- (l2);
\draw[connection] (l0) -- (l1);
\draw[connection] (l00) -- (l3);
\draw[connection] (l00) -- (l2);
\draw[connection] (l00) -- (l1);
\draw[connection] (l3) -- (l6);
\draw[connection] (l2) -- (l5);
\draw[connection] (l1) -- (l4);
\draw[connection] (l5) -- (l7);
\draw[connection] (l4) -- (l7);
\draw[connection] (l6) -- (l8);
\draw[connection] (l7) -- (l8);
\end{scope}
\end{tikzpicture}
\caption{A poset and its labelling.}
\label{fig:labelling}
\end{figure}
This labelling function is related to the measure of a poset, defined in \cite{MR294454}. The discussion of this relation is postponed to \cref{sect:measure}. Also, this labelling function induces a vanishing bound for the higher (co)limits of any functor as made precise by the following theorem.

\begin{thm}
\label{bound_by_B}
Let $\CP$ be a filtered poset, and ${B\colon \CP\to \NN}$ be its associated labelling function. Given a unital commutative ring $R$ and a functor $F\colon \CP^\op\to R\Mod$, we have:
\[
H^k(\CP;F)=0
\]
for every $k>\sup B$.
\end{thm}
\begin{proof}
We proceed by induction on the degree of $p$ to show that ${\alt(\NR F(p))\le B(p)}$ for all $p\in\CP$.

If $p\in \CP$ has degree $0$ or $1$, the result holds by definition.
For $p\in \CP$ of degree $n$, assume that $\alt(\NR F(s))\le B(s)$ for every $s\in \CP_{<p}$. By the induction hypothesis,
\[
\alt(\lim_{\CP_{<p}}\NR F)\le\max\{\alt(\NR F(s))\mid s<p\} \le \max\{B(s)\mid s<p\}=:m.
\]
If $\alt(\lim_{\CP_{<p}}\NR F)<m$ or $B(p)=m+1$, the result holds since
\[
\alt(\NR F(p)) \le \alt(\lim_{\CP_{<p}}\NR F)+1\le B(p).
\]
Therefore, we must only prove the result when $\alt(\lim_{\CP_{<p}}\NR F)=B(p)=m$. Notice that $\CB_p$ is an upper convex subposet of $\CP$, and for every $q<p$ with $q\not \in \CB_p$, we have $\NR F^k(q)=0$ for $k=m-1,m$. Therefore, for $k=m-1,m$, we have the following isomorphisms:
\[
\lim_{\CB_p}\NR F^{k}\cong \lim_{\CP_{<p}}\NR F^{k}.
\]
Next, since $B(p)=m$, $p$ does not close any circuit; thus, every connected component of $\CB_p$ has a single maximal element. Therefore, 
\[
\lim_{\CB_p}\NR F\cong {\bigoplus}_q\NR F(q)
\]
where $q$ ranges over the maximal elements of $\CP_{<p}$. Since every $\NR F(q)$ is a (possibly truncated) mapping cocylinder of the morphism $\varepsilon_q$, the last non-zero differential of the cochain complex $\bigoplus_q \NR F(q)$ is an epimorphism. Therefore, $\varepsilon_p$ is cotruncatable, and hence 
\[
\alt(\NR F(p))=\alt(\lim_{\CP_{<p}}\NR F)=m.
\]
Finally, we obtain the desired vanishing bound by applying Corollary~\ref{bound_from_height}.
\end{proof}

A \emph{filtered tree} is a filtered poset $\CP$ whose Hasse diagram contains no undirected cycles. If $\CP$ is a filtered tree, then no object $p\in \CP$ can close a circuit. Thus, we obtain the following corollary.

\begin{cor}
\label{cor:tree}
Let $\CP$ be a filtered tree. Then, for every functor $F\colon \CP^\op\to R\Mod$ and every $k>1$, we have
\[
H^k(\CP;F)=0.
\] 
\end{cor}

If $\CP$ is a filtered poset, a \emph{maximal tree} of $\CP$ is a wide subposet $\CT\le \CP$ such that $\CT$ is a tree, and for every subposet $\CP'$ with $\CT\subsetneq \CP'\subset \CP$, the poset $\CP'$ is not a tree. Given a poset $\CP$ and a maximal tree $\CT$, we define
\begin{equation*}
D(\CP,\CT)=\{d(q)\mid p\to q\in \CP\setminus \CT\}.
\end{equation*}
The next result uses \cref{bound_by_B} to give another geometric interpretation of a vanishing bound.
\begin{thm}
\label{Vanishingbouds/Tree/thm/bound_tree}
Let $\CP$ be a filtered poset, and $\CT$ be a maximal tree of $\CP$.
Then, for every functor $F\colon \CP^\op\to R\Mod$:
\[
H^k(\CP;F)=0
\]
for every $k> 2\# D(\CP,\CT)+1$.
\end{thm}
\begin{proof}
We prove instead that $\max B \le 2\# D(\CP,\CT)+1$ and we proceed by induction on $\# D(\CP,\CT)$. The basis case is when $\# D(\CP,\CT)=0$. In this case, $\CP$ is a tree and $\max B \le 1$, so the result follows.

Next, assume that the result holds for every poset such that $\#D(\CP',\CT')=n$ and let us prove it for $\#D(\CP,\CT)=n+1$.
Let $\CP$ and $\CT$ be as in the statement with $\#D(\CP,\CT)=n+1$. Let $k=\max D(\CP,\CT)$ and let $\CP'$ be the subposet of $\CP$ obtained by removing all relations $p\prec q \in \CP\setminus \CT$ such that $d(q)=k$. Denote by $B_{\CP'}$ the labelling function of $\CP'$.

Then $\#D(\CP',\CT)=\#D(\CP,\CT)-1=n$. By induction hypothesis, 
\[
{\max B_{\CP'} \le 2\# D(\CP',\CT)+1\le 2n+1}.
\]
Therefore, by definition of the labelling function, for every element $p\in \CP$ with $d(p)<k$, $B_{\CP'}(p)=B(p)\le 2n+1$. Then, for every relation $p\prec q\in \CP\setminus\CT$ with $d(q)=k$, we have 
\begin{equation}\label{eq:proof-missingarrow}
B(q)\le 2(n+1).
\end{equation}

Now, we prove that for every $p\in \CP$, $B(p)\le 2(n+1)+1$ by contradiction. Assume that there exists $p\in \CP$ such that $B(p)=2(n+1)+2$ and every element $p'\le p$ satisfies $B(p')<B(p)$.
Then, a connected component of $\CB_p$ has at least two maximal elements. But the elements in $\CB_p$ satisfy $B(q)$ equals either $2(n+1)$ or $2(n+1)+1$, see \eqref{eq:proof-missingarrow}. In both cases, this implies $\CB_p$ is a full subposet of $\CT$, which contradicts $\CT$ being a tree.
\end{proof}

Figure~\ref{fig:optimal_for_n} shows that the bound for $B$ in the proof above is sharp.

\begin{figure}[h!]
\centering
\begin{tikzpicture}[
x=1.5cm, y=1cm,
dot2/.style={circle, draw=darkslate, fill=vanilla!95!darkslate, ultra thick, inner sep=4pt, font=\bfseries},
dot/.style={circle, draw=darkslate, fill=vanilla!80!honey, ultra thick, inner sep=4pt, font=\bfseries},
dot3/.style={circle, draw=darkslate, fill=honey, ultra thick, inner sep=4pt, font=\bfseries},
connection/.style={draw=darkslate, thick}
]

\node[dot] (v00) at (-2, 0) {0};
\node[dot] (v10) at (-1, 0) {0};
\node[dot] (v11) at (-1, 1) {1};
\node[dot] (v21) at (0, 1) {1};
\node[dot] (v02) at (-2, 2) {1};
\node[dot] (v22) at (0, 2) {2};
\node[dot] (v13) at (-1, 3) {3};

\node[dot2] (v30) at (1, 0) {0};
\node[dot2] (v40) at (2, 0) {0};
\node[dot2] (v41) at (2, 1) {1};
\node[dot2] (v51) at (3, 1) {1};
\node[dot2] (v32) at (1, 2) {1};
\node[dot2] (v52) at (3, 2) {2};
\node[dot2] (v43) at (2, 3) {3};

\node[dot3] (v44) at (2, 4) {3};
\node[dot3] (v54) at (3, 4) {3};
\node[dot3] (v15) at (-1, 5) {3};
\node[dot3] (v55) at (3, 5) {4};
\node[dot3] (v36) at (1, 6) {5};

\draw[connection] (v00) -- (v02);
\draw[connection] (v10) -- (v11);
\draw[connection] (v10) -- (v21);
\draw[connection, dashed] (v11) -- (v02);
\draw[connection, dashed] (v11) -- (v22);
\draw[connection] (v21) -- (v22);
\draw[connection] (v02) -- (v13);
\draw[connection] (v22) -- (v13);

\draw[connection] (v30) -- (v32);
\draw[connection] (v40) -- (v41);
\draw[connection] (v40) -- (v51);
\draw[connection, dashed] (v41) -- (v32);
\draw[connection, dashed] (v41) -- (v52);
\draw[connection] (v51) -- (v52);
\draw[connection] (v32) -- (v43);
\draw[connection] (v52) -- (v43);

\draw[connection] (v43) -- (v44);
\draw[connection] (v43) -- (v54);
\draw[connection, dashed] (v44) -- (v15);
\draw[connection, dashed] (v44) -- (v55);
\draw[connection] (v54) -- (v55);
\draw[connection] (v13) -- (v15);
\draw[connection] (v15) -- (v36);
\draw[connection] (v55) -- (v36);

\end{tikzpicture}
\caption{}
\label{fig:optimal_for_n}
\end{figure}
By dualising the arguments above, we obtain the dual version of the preceding theorems.
\begin{thm}
\label{bound_by_B_colim}
Let $\CP$ be a filtered poset, and ${B\colon \CP\to \NN}$ be its associated labelling function. Given a unital commutative ring $R$ and a functor $F\colon \CP\to R\Mod$, we have:
\[
H_k(\CP;F)=0
\]
for every $k>\sup B$. \qed
\end{thm}
This yields the analogous corollary for filtered trees.

\begin{cor}
Let $\CP$ be a filtered tree. Then, for every functor $F\colon \CP\to R\Mod$, and every $k>1$, we have
\[
H_k(\CP;F)=0.
\] 
\end{cor}

As before, we recover the vanishing bound induced by the maximal tree.

\begin{thm}
\label{Vanishingbouds/Tree/thm/bound_tree_dual}
Let $\CP$ be a filtered poset, and $\CT$ be a maximal tree of $\CP$.
Then, for every functor  $F\colon \CP\to R\Mod$:
\[
H_k(\CP;F)=0
\]
for every $k> 2\# D(\CP,\CT)+1$.
\end{thm}

\subsection{Relation with the measure of a poset.}\label{sect:measure}
In this section we compare the labelling function $B$ described in Definition~\ref{labelling_function} with the measure of a poset, as defined by Mitchell in \cite{MR294454}.

Let $\CP$ be a filtered poset. We say that $p\in \CP$ is \emph{unicovered} if the set $\{q\in \CP \mid p\prec q\}$ has exactly one element.
Let $\CP$ be a filtered poset with an initial object $\hat 0$ and a terminal object $\hat 1$. We define the \emph{measure} of $\CP$, denoted by $\mu(\CP)$, inductively as follows. First, $\mu(\CP)=0$ if and only if $\CP$ has only one element. Now assume $n>0$ and that we have defined what we mean by $\mu(\CP)\le n-1$. An $n$-sequence in $\CP$ is a sequence of elements $p_1, p_2,\dots, p_k=\hat 0$, not necessarily ordered, such that, for each $i,\ 1\le i \le k$, we have:
\begin{enumerate}
\item $p_i$ is unicovered in $\CP\setminus\{p_1,\dots,p_{i-1}\}$, and
\item $\mu\Big((\CP\setminus\{p_1,\dots ,p_{i-1}\})_{\le p_i}\Big) \le n-1$.
\end{enumerate}
\begin{minipage}{0.7\textwidth}
The measure of $\CP$ is the smallest non-negative integer such that there exists an $n$-sequence in $\CP$. For a filtered poset $\CP$, one defines:
\[
\delta(\CP)=\sup\Big\{\mu(\CP_{\ge q}\cap \CP_{\le p})\mid q\le p \Big\}
\]
Notice that even if $\CP$ is a poset with an initial and a terminal object, $\delta(\CP)$ does not need to be equal to $\mu(\CP)$. For example, let $\CP$ be the poset whose Hasse diagram is drawn in Figure~\ref{fig:delta>mu}. The sequence consisting only of $\hat 0$ is a $1$-sequence, since $\hat 0$ is unicovered in $\CP$, and $\mu(\CP_{\le \hat 0}) = 0$. 
Hence $\mu(\CP)\le 1$. However, $\delta(\CP)\ge \mu(\CP_{\ge q}\cap \CP_{\le \hat 1})=2$.
\end{minipage}
\hfill
\begin{minipage}{0.25\textwidth}

\begin{tikzpicture}[scale=.8,
x=1.5cm, y=1.25cm,
dot2/.style={circle, draw=darkslate, fill=vanilla!95!darkslate, ultra thick, inner sep=4pt, font=\bfseries},
dot/.style={circle, draw=darkslate, fill=vanilla!80!honey, ultra thick, inner sep=4pt, font=\bfseries},
dot3/.style={circle, draw=darkslate, fill=honey, ultra thick, inner sep=4pt, font=\bfseries},
connection/.style={draw=darkslate, thick}
]

\node[dot] (00) at (0, 0) {$\hat{0}$};
\node[dot] (q) at (0, 1) {$q$};
\node[dot] (r2) at (1, 2) {$r_2$};
\node[dot] (r1) at (-1, 2) {$r_1$};
\node[dot] (p) at (0, 3) {$\hat{1}$};

\draw[connection] (00) -- (q);
\draw[connection] (q) -- (r2);
\draw[connection] (q) -- (r1);
\draw[connection] (r2) -- (p);
\draw[connection] (r1) -- (p);
\end{tikzpicture}
\captionof{figure}{}
\label{fig:delta>mu}
\end{minipage}

\begin{thm}
\label{thm:delta<=B}
Let $\CP$ be a filtered poset, $B\colon \CP\to \NN$ the labelling function of $\CP$. Then,
\[
\delta(\CP)\le \sup_{p\in \CP} B(p).
\]
\end{thm}
The proof of this theorem relies on the following auxiliary lemma and Remark~\ref{rmk:B_q_1 vs B_q_2}.
\begin{lem}
\label{lem:mu<=B}
Let $\CP$ be a filtered poset with initial and terminal objects, $\hat 0$ and $\hat 1$ respectively. Then,
\[
\mu(\CP)\le B(\hat 1)
\]
\end{lem}
\begin{proof}
We proceed by induction on $B(\hat 1)$. If $B(\hat 1)=0$, then $\CP$ has only one element, so $\mu(\CP)=0$. 
If $B(\hat 1)=1$, then $\CP$ is a total order, so $\mu(\CP)=1$. Assume the result holds for every poset such that $B(\hat 1)<m$ and let $\CP$ be a poset with $B(\hat 1)=m$.

We recursively define $Q_0=\varnothing$; and, for $i\ge 1$,
\[
Q_i=\max\left\{q\in \CP\setminus\left({\bigcup}_{j<i} Q_j\right)\Big\vert B(q)=m-1\right\},
\]
and choose any ordering of its elements $q_1^i,q_2^i,\dots, q_{r_i}^i$. Since $\CP$ is a filtered poset, $\CB_{\hat 1}$ is finite, hence there exists $k\in \NN$ such that $Q_j=\emptyset$ for every $j>k$. Next, let 
\[
{\CP'=\CP\setminus\left(\bigcup_{j\le k} Q_j\right)},
\]
and $B'$ be its associated labelling function. By construction, ${B'(\hat 1)\le m-1}$. Therefore, by the induction hypothesis, $\mu(\CP')\le m-1$, so there exists an $(m-1)$-sequence in $\CP'$, call it $p_1,\dots,p_{r'}=\hat 0$.

Then, we claim that the sequence:
\[
q_1^1,q_2^1,\dots,q_{r_1}^1,q_1^2,\dots, q_i^j,\dots, q_{r_k}^k,p_1,\dots, p_{r'}=\hat 0
\]
is an $m$-sequence in $\CP$. First, for the elements $q_i^j$:
\begin{enumerate}
\item $q_i^j$ is unicovered in $\CP\setminus\{q_1^1,\dots,q_{i-1}^j\}$:
\begin{enumerate}
\item If $\hat 1$ closes a circuit in $\CB_{\hat 1}$, then $q_i^j$ is unicovered in $\CP\setminus\{q_1^1,\dots,q_{i-1}^j\}$ because $q_i^j\prec \hat 1$ in $\CP\setminus\{q_1^1,\dots, q_{r_{j-1}}^{j-1}\}$.
\item Otherwise, assume by contradiction that $\hat 1$ does not close a circuit in $\CB_{\hat 1}$ and $q_i^j$ is not unicovered in $\CP\setminus\{q_1^1,\dots,q_{i-1}^j\}$. Then, there exist $\tilde p_1,\tilde p_2$ such that  $q_i^j<\tilde p_1,\tilde p_2<\hat 1$. But, by construction of the sequence, $B(\tilde p_1)=B(\tilde p_2)=m$.  Then, $B(\hat 1)$ could not be $m$ and this is a contradiction.
\end{enumerate} 
\item  $\mu\left(\left(\CP\setminus\{q_1^1,\dots,q_{i-1}^j\}\right)_{\le q_i^j}\right)= \mu(\CP_{\le q_i^j})\le B(q_i^j)=m-1$ by the induction hypothesis.
\end{enumerate}
Finally, for the elements $p_i$, both conditions hold since 
\[
\CP\setminus\{q_1^1,\dots,q_{r_k}^k,p_1,\dots,p_{i-1}\}=\CP'\setminus\{p_1,\dots,p_{i-1}\}
\]
and $p_1,\dots,p_{r'}=\hat 0$ is an $(m-1)$-sequence in $\CP'$.
\end{proof}
\begin{rmk}
\label{rmk:B_q_1 vs B_q_2}
Let $\CP$ be a filtered poset, and  let $B_{q,p}$ be the associated labelling function for the subposet $\CP_{\ge q}\cap \CP_{\le p}$, for every $p\le q\in \CP$. Given $q_1\le q_2\le p$, the circuits present in $\CP_{\ge q_2}\cap \CP_{\le p}$ are also present in $\CP$ and in $\CP_{\ge q_1}\cap \CP_{\le p}$. Thus, we have:
\[
B(r)\ge B_{q_1,p}(r)\ge B_{q_2,p}(r),
\]
for every $r\in \CP$ such that $q_2\le r\le p$.
\end{rmk}
\begin{proof}[Proof of \cref{thm:delta<=B}.]
For every comparable pair $q\le p$ in $\CP$, the interval $\CP_{\ge q}\cap \CP_{\le p}$ is a filtered poset with initial object $q$ and terminal object $p$. Let $B_{q,p}$ be its associated labelling function. By Lemma~\ref{lem:mu<=B}, we have $\mu(\CP_{\ge q}\cap \CP_{\le p})\le B_{q,p}(p)$. 
By Remark~\ref{rmk:B_q_1 vs B_q_2}, we have $B_{q,p}(p)\le B(p)$. Thus, we obtain
\[
\mu(\CP_{\ge q}\cap \CP_{\le p})\le B_{q,p}(p)\le B(p).
\]
Taking the supremum over all comparable pairs $q\le p$, we conclude:
\begin{align*}
\delta(\CP)=&\sup\{\mu(\CP_{\ge q}\cap \CP_{\le p})\mid q\le p\}
\le \sup\{B_{q,p}(p) \mid q\le p\}  \\
\le&\sup \{B(p)\mid p\in \CP\}. \qedhere
\end{align*}
\end{proof}

\begin{figure}[h!]
\centering
\begin{tikzpicture}[
x=1cm, y=.5cm,
dot2/.style={circle, draw=darkslate, fill=vanilla!95!darkslate, ultra thick, inner sep=4pt, font=\bfseries},
dot/.style={circle, draw=darkslate, fill=vanilla!80!honey, ultra thick, inner sep=4pt},
dot3/.style={circle, draw=darkslate, fill=honey, ultra thick, inner sep=4pt, font=\bfseries},
connection/.style={draw=darkslate, thick}
]

\node[dot] (000) at (1, 0) {};
\node[dot] (00) at (-1, 0) {};

\node[dot] (10) at (-2, 2) {};
\node[dot] (12) at (2, 2) {};

\node[dot] (20) at (-2, 4) {};
\node[dot] (21) at (0, 3) {};
\node[dot] (22) at (2, 4) {};

\node[dot] (30) at (-3, 6) {};
\node[dot] (31) at (-1, 6) {};
\node[dot] (32) at (0, 6.5) {};
\node[dot] (33) at (1, 6) {};
\node[dot] (34) at (3, 6) {};

\node[dot] (40) at (-2, 8) {};
\node[dot] (41) at (2, 8) {};

\node[dot] (111) at (0, 10) {};

\draw[connection] (00) -- (10);
\draw[connection] (00) -- (21);
\draw[connection] (00) -- (12);
\draw[connection] (000) -- (10);
\draw[connection] (000) -- (21);
\draw[connection] (000) -- (12);

\draw[connection] (10) -- (20);
\draw[connection] (12) -- (22);

\draw[connection] (20) -- (30);
\draw[connection] (20) -- (31);
\draw[connection] (21) -- (32);
\draw[connection] (22) -- (33);
\draw[connection] (22) -- (34);

\draw[connection] (30) -- (40);
\draw[connection] (31) -- (40);
\draw[connection] (32) -- (40);
\draw[connection] (32) -- (41);
\draw[connection] (33) -- (41);
\draw[connection] (34) -- (41);

\draw[connection] (40) -- (111);
\draw[connection] (41) -- (111);
\end{tikzpicture}
\caption{}\label{Fig:delta<B}
\end{figure}
\begin{rmk}
The equality $\delta(\CP)= \sup\{B(q)\mid q\in \CP\}$ does not hold in general and the poset $\CP$ whose Hasse diagram is drawn in Figure~\ref{Fig:delta<B} is an example of this. Moreover, despite $\delta(\CP)$ being a vanishing bound for the higher limits of any functor, it is not always a bound for the height of its fibrant replacement: The constant functor over the poset $\CP$ is also an example of this.
\end{rmk}

\subsection{Application to Hypergraph Cohomology}
Hypergraphs generalise graphs by allowing edges—called \emph{hyperedges}—to contain an arbitrary number of vertices, not just two. This extra flexibility captures higher-order incidence relations that cannot be faithfully encoded by a graph. 
\begin{defn}
An undirected hypergraph is a tuple $\CH=(V,E)$ where $V$ is a finite set and $E$ is a collection of subsets of $V$. The elements of $V$ are called vertices (or nodes), and the elements of $E$ are called hyperedges.
\end{defn}
Notice that we do not allow multiple hyperedges between the same objects, i.e., the set of edges $E$ is actually a set of sets, not a multiset of sets.

\begin{figure}[h!]
\centering
\begin{tikzpicture}[node distance=60pt]
\node[circle, fill=vanilla, draw=darkslate, ultra thick, inner sep=4pt, label=above:\(v_1\)] (v1) {};
\node[circle, fill=vanilla, draw=darkslate, ultra thick, inner sep=4pt, right of=v1, label=above:\(v_2\)] (v2) {};
\node[circle, fill=vanilla, draw=darkslate, ultra thick, inner sep=4pt, right of=v2, label=above:\(v_3\)] (v3) {};
\node[circle, fill=vanilla, draw=darkslate, ultra thick, inner sep=4pt, below of=v1, label=below:\(v_4\)] (v4) {};
\node[circle, fill=vanilla, draw=darkslate, ultra thick, inner sep=4pt, right of=v4, label=below:\(v_5\)] (v5) {};
\node[circle, fill=vanilla, draw=darkslate, ultra thick, inner sep=4pt, right of=v5, label=below:\(v_6\)] (v6) {};
\node[circle, fill=vanilla, draw=darkslate, ultra thick, inner sep=4pt, below of=v5, label=below:\(v_7\)] (v7) {};
\begin{pgfonlayer}{background}
\tikzset{edge/.style={line cap=round, line join=round, opacity=0.5}}
\begin{scope}[transparency group, opacity=.5]   
\draw[line cap=round, line join=round, line width=50pt, color=darkslate, fill=darkslate] (v4.center) -- (v1.center) -- (v2.center) -- (v3.center) -- cycle;
\end{scope}
\begin{scope}[transparency group, opacity=.5]
\draw[edge, line width=50pt, color=honey] (v3.center) -- (v5.center) -- (v6.center) -- cycle;
\end{scope}
\draw[edge, line width=50pt, color=brownred] (v1.center) -- (v2.center);
\draw[edge, line width=50pt, color=brownred] (v2.center) -- (v5.center);
\begin{scope}[transparency group, opacity=.5]   
\draw[edge, line width=50pt, color=honey] (v4.center) -- (v5.center) -- (v7.center) --cycle;
\end{scope}
\end{pgfonlayer}
\end{tikzpicture}
\caption{A hypergraph}\label{fig:hipergraph}   
\end{figure}
In a graph, we represent vertices as points and edges as lines connecting them. In a hypergraph, we also represent vertices as points and edges as closed loops or areas enclosing the vertices they contain; see Figure~\ref{fig:hipergraph}. Observe that graphs may be viewed as $1$-dimensional simplicial complexes. In contrast, for a general hypergraph, subedges of an $n$-edge need not belong to the hypergraph. Nevertheless, one can still define the associated face poset.

\begin{defn}
Given a hypergraph $\CH=(V,E)$, we define the \emph{face poset} $\CP(\CH)$ to be the poset whose underlying set is $V\cup E$ and whose order is given by inclusion. 
The simplicial closure of the hypergraph, denoted by $\Delta(\CH)$, is the simplicial complex obtained by closing $V\cup E$ under taking subsets.
\end{defn}

There are multiple ways of defining the cohomology of a hypergraph (see \cite{GPSZ26} for a beautiful overview), but for us, the most natural one is to consider the face poset of the hypergraph. A decisive structural advantage of this approach is that $\CP(\CH)$ records only the hyperedges actually present in $\CH$, whereas the simplicial closure $\Delta(\CH)$ artificially adds all sub-simplices, inflating the poset and distorting the combinatorics. In particular, the length of $\CP(\CH)$ can be substantially smaller than that of $\CP(\Delta(\CH))$, see Figure~\ref{fig:hypergraph_intro_comparison}.
\begin{figure}[h!]
\centering
\begin{tikzpicture}[
poset_node/.style={circle, draw=darkslate!80, fill=darkslate!10, thick, inner sep=1.5pt, minimum size=6mm},
connection/.style={-, thick, darkslate!70},
vertex/.style={circle, draw=darkslate, fill=darkslate, inner sep=1.5pt}
]

\begin{scope}[shift={(0,0)}]
\node (title) at (0, -2.2) {\small \textbf{Hypergraph} $\mathcal{H}$};

\coordinate (A) at (-0.8, -0.6); 
\coordinate (B) at (1.0, -0.4);  
\coordinate (C) at (0, 1.2);     
\coordinate (D) at (0.3, 0.2);   

\filldraw[honey!30, draw=none, rounded corners=15pt, fill opacity=0.6] 
(-1.2,-1.0) -- (1.4,-0.8) -- (0.5, 1.6) -- (-1.1, 1.2) -- cycle;

\draw[thick, darkslate] (A) -- (B);

\node[vertex, label=left:{\footnotesize $2$}] (v2) at (A) {};
\node[vertex, label=right:{\footnotesize $3$}] (v3) at (B) {};
\node[vertex, label=above:{\footnotesize $1$}] (v1) at (C) {};
\node[vertex, label=below:{\footnotesize $4$}] (v4) at (D) {};
\end{scope}

\begin{scope}[shift={(4.5,0)}]
\node (title) at (0, -2.2) {\small \textbf{Face Poset} $\mathcal{F}(\mathcal{H})$};

\node[poset_node] (n1) at (-1.2, -0.8) {\footnotesize $1$};
\node[poset_node] (n2) at (-0.4, -0.8) {\footnotesize $2$};
\node[poset_node] (n3) at (0.4, -0.8) {\footnotesize $3$};
\node[poset_node] (n4) at (1.2, -0.8) {\footnotesize $4$};

\node[poset_node] (n23) at (0, 0.3) {\footnotesize $23$};
\node[poset_node] (n1234) at (0, 1.5) {\footnotesize $1234$};

\draw[connection] (n2) -- (n23);
\draw[connection] (n3) -- (n23);

\draw[connection] (n1) -- (n1234);
\draw[connection] (n23) -- (n1234);
\draw[connection] (n4) -- (n1234);
\end{scope}

\begin{scope}[shift={(9,0)}]
\node (title) at (0, -2.2) {\small \textbf{Simplicial Closure} $\Delta(\mathcal{H})$};

\coordinate (A) at (-0.8, -0.6); 
\coordinate (B) at (1.0, -0.4);  
\coordinate (C) at (0, 1.2);     
\coordinate (D) at (0.3, 0.2);   

\fill[honey!50, opacity=0.8] (A) -- (B) -- (D) -- cycle;
\fill[honey!30, opacity=0.8] (B) -- (C) -- (D) -- cycle;
\fill[honey!70, opacity=0.8] (C) -- (A) -- (D) -- cycle;

\draw[thick, darkslate] (A) -- (B) -- (C) -- cycle;
\draw[thick, darkslate] (A) -- (D);
\draw[thick, darkslate] (B) -- (D);
\draw[thick, darkslate] (C) -- (D);

\node[vertex, label=left:{\footnotesize $2$}] (v2) at (A) {};
\node[vertex, label=right:{\footnotesize $3$}] (v3) at (B) {};
\node[vertex, label=above:{\footnotesize $1$}] (v1) at (C) {};
\node[vertex, label=below:{\footnotesize $4$}] (v4) at (D) {};
\end{scope}
\end{tikzpicture}
\caption{Comparison of a hypergraph $\mathcal{H}$ with its Hasse diagram for the face poset $\mathcal{P}(\mathcal{H})$, and its simplicial closure $\Delta(\mathcal{H})$.}
\label{fig:hypergraph_intro_comparison}
\end{figure}

Hypergraphs are particularly relevant in network science, where a \emph{network} is the mathematical model of a system of entities and their interactions. Classical network models rely on graphs, where nodes represent entities and edges encode pairwise relations. However, many real-world systems exhibit \emph{higher-order interactions}: in a co-authorship network, nodes are researchers and a hyperedge groups all co-authors of a single article; in a plant-pollinator network, nodes are species and each pollination event—involving several plants and pollinators at once—is naturally encoded as a hyperedge. Such collective interactions cannot be faithfully reduced to pairwise edges without loss of combinatorial information, making hypergraphs the appropriate structure. Functors $F\colon \CP(\CH)^\op\to R\Mod$ over the face poset of a hypergraph provide a flexible algebraic framework for studying such data; in the applied literature these are often called \emph{sheaves} \cite{Curry14, HansenGhrist2021, Robinson2017}, a terminology that underscores the local-to-global character of the construction. Going beyond the constant functor is particularly relevant, as non-constant sheaves can encode local weights or orientations associated with each hyperedge.

Below we illustrate this framework on two empirical datasets: a co-authorship hypergraph from the \texttt{arXiv} metadata and a plant-pollinator interaction dataset. In both cases we compute $\sup B$ algorithmically and compare it with the classical bound given by the length of the poset, showing that \cref{bound_by_B} gives a strictly sharper result.

\begin{exm}
\label{exm:hypergraph-arxiv-kaggle}
Let $\mathcal{H} = (V, E)$ be a co-authorship hypergraph constructed from the \texttt{arXiv} metadata dataset hosted on Kaggle \cite{arxiv_kaggle_dataset}. For this analysis, the network was processed using the \texttt{XGI} library \cite{xgi2023}, extracting a sub-network consisting of the first $7,000$ interactions (articles). The resulting structure comprises $\#V = 18,127$ nodes, with a maximum hyperedge size of $99$, reflecting the presence of large-scale scientific collaborations.

To investigate the combinatorics of the hypergraph, we constructed the face poset $\CP(\CH)$. The system exhibits a poset height of $3$. Therefore, it is natural to expect non-zero cohomology groups in dimensions $0$, $1$, $2$, and $3$.
For computational reasons, we compute the cohomology of the constant functor with coefficients in the field $\QQ$. We obtain that the third cohomology group vanishes. This shows that the vanishing bound provided by the length of the poset is not optimal.

We implemented the labelling function $B$ in Sagemath, see \cite{Carrion-Github}, for the face poset associated with the sampled co-authorship hypergraph. The computation gives $\sup B=2$.
Therefore, \cref{bound_by_B} yields
\[
H^k(\CP;F)=0\quad\text{for every }k>2
\]
for any functor $F\colon \CP^\op\to R\Mod$. In this case, the bound is sharp: for the constant functor, one still detects nontrivial cohomology in degree $2$, while higher degrees vanish.
\end{exm}

\begin{exm}
\label{exm:hypergraph-plant-pollinator}
To further illustrate the applicability of our bound, we performed an analogous analysis on a different ecological context: the plant-pollinator interaction dataset \texttt{plant-pollinator-mpl-062} (see \cite{lucas_2024_13753744}). The computation processed the network, which consists of $456$ nodes and $1044$ interactions. The resulting structure reveals a complex organizational pattern, characterized by a maximum hyperedge dimension of $157$, reflecting the presence of highly multi-species pollination events.

We constructed the face poset $\CP(\CH)$. The system exhibits a poset height of $4$, establishing a theoretical framework for non-trivial cohomology groups up to degree $4$. Furthermore, the maximum value of the labelling function is $\max B = 3$, providing an additional bound on the dimensions where topological features may manifest.

For the constant sheaf with coefficients in the rationals $\mathbb{Q}$, the computed cohomology groups reveal a rich structure in the lower dimensions:
\[
\dim H^0(\CP(\CH);\underline \QQ)=1,\ \dim H^1(\CP(\CH);\underline \QQ)=12642, \text{ and } \dim H^2(\CP(\CH);\underline \QQ)=32.
\]
Notably, the cohomology groups vanish for $k \ge 3$. This behavior is consistent with the bound induced by $\sup B = 3$, demonstrating how the intrinsic properties of the labelling function can refine the expectations derived solely from the poset height.
\end{exm}

\section{Inductive bound for the vanishing of the higher limits}
\label{sec:Inductive_bound}

While the combinatorial bound of \cref{sect:combinatorial_bound} depends only on the geometry of the underlying poset, the vanishing bounds introduced in this section are of a different nature: they are driven by the local algebraic information of the functor.  We first establish the result for higher limits; the dual statement for higher colimits is postponed to the end of the section.

The inductive vanishing bound presented here exploits the fact that the cotruncatability of the natural morphism
\[
F(p)\longrightarrow \slim_{\CP_{<p}}\NR F
\]
is controlled by the vanishing of the higher limits of $F$ restricted to the full subcategories $\CP_{<p}$: $p$ is $F$-truncatable if and only if $\alt(\lim_{\CP_{<p}}\NR F)=n>0$, and the last non-trivial differential of $\lim_{\CP_{<p}}\NR F$,
\[
\partial\colon(\lim_{\CP_{<p}}\NR F)^{n-1}\longrightarrow (\lim_{\CP_{<p}}\NR F)^{n},
\]
is an epimorphism. Moreover, the $n$th higher limit of $F$ restricted to $\CP_{<p}$ is the cokernel of $\partial$. This is the main ingredient for the following result.

\begin{thm}
\label{Thm-inductive1}
Let $\CP$ be a filtered poset and $R$ be a unital commutative ring. Let ${F\colon \CP^\op\to R\Mod}$ be a functor. If there exists an integer $n\ge 0$ such that for every $p\in \CP$, $H^k(\CP_{< p};F)=0$ for all $k\ge n$, then
\[
H^k(\CP;F)=0
\]
for all $k>n$.
\end{thm}

\begin{exm}
\label{exm:CW-complex}
If $\CP$ is the face poset of a regular CW-complex $X$, $\CP_{<p}$ corresponds to the $(d(p)-1)$-skeleton of the boundary of the $d(p)$-cell $p$. Then, if $\underline \ZZ$ is the constant functor, $H^k(\CP_{<p};\underline \ZZ)\cong H^k(\partial p;\ZZ)=0$ if $k\ge d(p)$. Therefore, \cref{Thm-inductive1} recovers the well-known bound for the cohomology of CW-complexes 
\[
H^k(X;\ZZ)=H^k(\CP;\underline \ZZ)=0
\]
if $k>\dim X=\max \{d(p)\mid p\in \CP\}$.
\end{exm}

To prove \cref{Thm-inductive1} we first need the following auxiliary lemmas.

\begin{prop}
\label{no jump property}
Let $\CP$ be a filtered poset and $R$ be a unital commutative ring. Let ${F\colon \CP^\op\to R\Mod}$ be a functor. Then, for every $0\le n\le \alt(\NR F)$, there exists $p\in \CP$ such that $\alt(\NR F(p))=n$. Furthermore, if $n>0$, there exists $p\in \CP$ such that $\alt(\lim_{\CP_{<p}}\NR F)=n-1$.
\end{prop}
\begin{proof} 
For brevity, set $\ell=\alt(\NR F)$. If $n=0$ or $n=\ell$, the result is trivial. Assume by contradiction that there exists $n\in \NN$ with $0< n<\ell$ such that $\alt(\NR F(p))\neq n$ for all $p\in \CP$. Let $Q$ be the subposet of $\CP$ given by:
\[
Q=\{q\in \CP\mid \alt(\NR F(q))>n\}.
\]
By hypothesis, $Q$ is non-empty because $\alt(\NR F)=\ell>n$, so there exists some $q\in \CP$ such that $\alt(\NR F(q))=\ell$. 

Now, let $q$ be a minimal object in $Q$, that is, ${\alt(\NR F(q))>n}$, and for every $t<q$, $\alt(\NR F(t))< n$. This implies that ${\alt(\lim_{\CP_{<q}}\NR F)<n}$ because
\[
\alt(\lim_{\CP_{<q}}\NR F) \le \max\{\alt(\NR F(t))\mid t<q\}<n.
\]
If $q$ is $F$-truncatable, we obtain the following contradiction:
\[ 
n<\alt(\NR F(q))=\alt(\cocyl_{\NT}(F(q)\to \lim_{\CP_{<q}}\NR F))=\alt(\lim_{\CP_{<q}}\NR F)<n.
\]
Thus, $q$ is not $F$-truncatable. But this also leads to the following contradiction:
\[
n<\alt(\NR F(q))=\alt(\cocyl(F(q)\to \lim_{\CP_{<q}}\NR F))=\alt(\lim_{\CP_{<q}}\NR F)+1 \le n.
\]
Thus, there exists at least one $p\in \CP$ such that $\alt(\NR F(p))=n$. Moreover, if such $p\in \CP$ is minimal satisfying $\alt(\NR F(p))=n$, then $\alt(\lim_{\CP_{<p}}\NR F)=n-1$.
\end{proof}

\begin{lem}
\label{prop_VB_inductive equivalences}
Let $\CP$ be a filtered poset, and $R$ be a unital commutative ring. Let ${F\colon \CP^\op\to R\Mod}$ be a functor, and $n$ be a positive integer. The following are equivalent:
\begin{enumerate}
\item $H^k(\CP_{< p};F)=0$, for every $p\in \CP$ and $k\ge n$.
\item $H^n(\CP_{< p};F)=0$, for every $p\in \CP$.
\item The fibrant replacement $\NR F$ has height $\alt(\NR F)\le n$.
\end{enumerate}
\end{lem}
\begin{proof}
(2) is a special case of (1), so (1)$\Rightarrow$(2). For (2)$\Rightarrow$(3), we proceed by contradiction. Assume that $\alt(\NR F)>n>0$. By Proposition~\ref{no jump property}, there exists an object $p\in \CP$ such that $\alt(\NR F(p))=n+1$ and $\alt(\lim_{\CP_{<p}}\NR F)=n$.
Then, by  the inductive construction of the fibrant replacement:
\[
H^n(\CP_{< p};F)=H^n(\lim_{\CP_{<p}}\NR F)=\coker(\partial_{n-1}\colon(\lim_{\CP_{<p}}\NR F)^{n-1}\to(\lim_{\CP_{<p}}\NR F)^n).
\]
However, $H^n(\CP_{< p};F)=0$ if and only if the differential $\partial_{n-1}$ is an epimorphism and this is equivalent to the fact that the morphism ${\varepsilon_p\colon F(p)\to \lim_{\CP_{<p}}\NR F}$ is cotruncatable. Thus, by definition,  $\NR F(p)=\cocyl_\NT(\varepsilon_p)$.
Therefore, $\alt(\NR F(p))=n\neq n+1$ which is a contradiction.

Finally, we must prove (3)$\Rightarrow$(1). Notice that (1) is equivalent to,
\begin{center}
for every $p\in \CP$, and $k\ge n$, $H^k(\lim_{\CP_{<p}}\NR F)=0$. 
\end{center}
So, fix $p\in \CP$. By hypothesis, $\alt(\NR F(p))\le n$, so we only need to prove that 
\[
{H^n(\lim_{\CP_{<p}}\NR F)=0}.
\]
There is no loss of generality in assuming that ${\alt(\lim_{\CP_{<p}}\NR  F)=n}$; otherwise, there is nothing to prove. In that case, the natural map $\varepsilon_p\colon F(p)\to \lim_{\CP_{<p}}\NR F$ must be cotruncatable, and hence the differential $\partial_{n-1}\colon(\lim_{\CP_{<p}}\NR F)^{n-1}\to(\lim_{\CP_{<p}}\NR F)^n$ is an epimorphism. Thus, 
\[
H^n(\CP_{< p};F)=H^n(\lim_{\CP_{<p}}\NR F)=\coker(\partial_{n-1})=0.\qedhere
\]
\end{proof}

\begin{proof}[Proof of \cref{Thm-inductive1}] 
Now, the proof follows by combining Lemma~\ref{prop_VB_inductive equivalences} and Corollary~\ref{bound_from_height}.
\end{proof}

If $\CP$ is a filtered poset of finite length, we can relax the hypothesis of \cref{Thm-inductive1} to obtain a weaker vanishing bound for the higher limits of $F$, requiring the local vanishing condition only for $p$ of certain consecutive degrees.
\begin{thm}
\label{Thm-inductive2}
Let $\CP$ be a filtered poset of finite length, $R$ be a unital commutative ring, and $F\colon \CP^\op\to R\Mod$ be a functor. If there exist $m\le \operatorname{length}(\CP)$ and $n\in \NN$ such that, for every $p\in \CP$ with ${d}(p)\le m$, we have:
\[
H^k(\CP_{<p};F)=0\quad \mbox{for }n\le k.
\] 
Then, $H^k(\CP;F)=0$ for $n+\operatorname{length}(\CP)-m<k$.
\end{thm}
\begin{proof}
Let $\NR' F\colon \CP^\op\to \Ch(R)$ be the functor defined by:
\[
\NR' F(p)=\begin{cases}
\NR F(p) &\text{if }{d}(p)\le m,\\
\cocyl(F(p)\to \lim_{\CP_{<p}}\NR' F)& \text{otherwise,}
\end{cases}
\]
which is a fibrant replacement of $F$ by Proposition~\ref{thm:model_category}, see Remark~\ref{rmk:fib-choice-cocyl}. By Lemma~\ref{prop_VB_inductive equivalences}, for every $p\in \CP$ with ${d}(p)\le m$, we have:
\[
\alt(\NR' F(p))\le n.
\]
Now, we prove by induction on $k\ge 0$ that, for every object $p\in \CP$ with ${d}(p)=m+k$, we have
\[
\alt(\NR' F(p))\le n+k.
\]
The basis case is already done. Now, assume the result is true for objects of degree $m+k-1$, and let $p\in \CP$ with ${d}(p)=m+k$. Then,
\begin{align*}
\alt(\NR' F(p))=&\alt(\cocyl(F(p)\to \lim_{\CP_{<p}}\NR' F))\\
\le& \alt(\lim_{\CP_{<p}}\NR' F)+1\le \max \{\alt(\NR' F(q))\mid q\le p\}+1\\
\le& n+k-1+1=n+k.
\end{align*}
Then, $\alt(\NR' F)\le n+\operatorname{length}(\CP)-m$ which implies the desired vanishing bound.
\end{proof}
Finally, we present the dual results for the higher colimits without proofs.

\begin{thm}
\label{Thm-inductive1-dual}
Let $\CP$ be a filtered poset and $R$ be a unital commutative ring. Let ${F\colon \CP\to R\Mod}$ be a functor. If there exists an integer $n\ge 0$ such that for every $p\in \CP$, $H_k(\CP_{< p};F)=0$ for all $k\ge n$, then
\[
H_k(\CP;F)=0
\]
for all $k>n$.
\end{thm}

\begin{thm}
\label{Thm-inductive2-dual}
Let $\CP$ be a filtered poset of finite length, $R$ be a unital commutative ring, and $F\colon \CP\to R\Mod$ be a functor. If there exist $m\le \operatorname{length}(\CP)$ and $n\in \NN$ such that, for every $p\in \CP$ with ${d}(p)\le m$, we have:
\[
H_k(\CP_{<p};F)=0\quad \mbox{for }n\le k.
\] 
Then, $H_k(\CP;F)=0$ for $n+\operatorname{length}(\CP)-m<k$.
\end{thm}

\subsection{Application to Mackey Functors for Posets}
\label{sec:weak_mackey_functors}
Weak Mackey functors for posets were introduced by the author in \cite{CD23}, modelled on the classical notion of Mackey functor for orbit categories \cite{MR1759612,Jackowski1992}. The central result of that paper is that every weak Mackey functor with quasi-unit is acyclic, i.e., its higher limits (and higher colimits) vanish. The present section pursues a quantitative refinement of this statement: rather than requiring the quasi-unit condition throughout the entire poset, we impose it only on a partial range of degrees and appeal to \cref{Thm-inductive2} and \cref{Thm-inductive2-dual} to obtain a non-trivial vanishing range for the higher limits. We first recall the relevant definitions before establishing the main result.

\begin{defn}
\label{def:F-linear}
Let $F\colon \CP \to R\Mod$ be a functor and $p\in \CP$. An endomorphism $\alpha\in \End_R(F(p))$ is \emph{$F$-linear} if for every $q\le p$ there exists $\beta\in \End_R(F(q))$ such that
\[
\alpha\circ F(q\le p)=F(q\le p)\circ \beta.
\]
If, moreover, $\alpha$ and $\beta$ are automorphisms, we call $\alpha$ an \emph{$F$-linear automorphism}. We denote by $\End^F_R(p)$ and $\Aut^F_R(p)$ the submonoid and subgroup of $F$-linear endomorphisms and automorphisms of $F(p)$, respectively.

Dually, for a contravariant functor $G\colon \CP^\op \to R\Mod$ and $p\in \CP$, an endomorphism $\alpha\in \End_R(G(p))$ is \emph{$G$-linear} if for every $q<p$ there exists $\beta\in \End_R(G(q))$ such that
\[
G(q<p)\circ\alpha = \beta \circ G(q<p).
\]
We denote the corresponding submonoid and subgroup by $\End^G_R(p)$ and $\Aut^G_R(p)$.
\end{defn}

\begin{defn}
\label{def:weak-mackey}
Let $\CP$ be a filtered poset.
\begin{enumerate}
\item A covariant functor $F\colon \CP\to R\Mod$ is a \emph{weak Mackey functor} if for every $q\le p$ in $\CP$ there exists a morphism $G(q\le p)\colon F(p)\to F(q)$ such that 
\[
G(q\le p)\circ F(q\le p)=\alpha(p,q)
\]
with $\alpha(p,q)\in \End^F_{R}(q)$, and, for every $k<p$ with $q\not\le k$,
\[
\im\bigl(G(q\le p)\circ F(k\le p)\bigr)\subseteq \im_F(q),
\]
where $\im_F(q)=\sum_{k<q}\im(F(k\le q))$.
\item A contravariant functor $G\colon \CP^\op\to R\Mod$ is a \emph{weak Mackey functor} if for every $q<p$ in $\CP$ there exists a morphism $F(q<p)\colon G(q)\to G(p)$ such that 
\[
G(q<p)\circ F(q<p)=\alpha(p,q)
\]
with $\alpha(p,q)\in \End^G_{R}(q)$, and, for every $k<p$ with $q\not< k$,
\[
\ker_G(k)\subseteq \ker\bigl(G(q<p)\circ F(k<p)\bigr),
\]
where $\ker_G(k)=\bigcap_{l<k}\ker(G(l<k))$.
\end{enumerate}
Given a subposet $\CQ\subseteq \CP$, we say that $F$ (resp.\ $G$) has a \emph{quasi-unit in $\CQ$} if $\alpha(p,q)\in \Aut^F_{R}(q)$ (resp.\ $\alpha(p,q)\in \Aut^G_{R}(q)$) for all $q<p$ with $p,q\in \CQ$. When $\CQ=\CP$ we simply say that $F$ (resp.\ $G$) has a \emph{quasi-unit}.
\end{defn}

It is proved in \cite[Theorem C, C*]{CD23} that a weak Mackey functor with quasi-unit is $\lim$-acyclic (resp.\ $\colim$-acyclic). The following result provides a partial version of this statement, replacing the global quasi-unit condition by a degree-wise one.
\begin{thm}
\label{thm:VB_inductive_weak_mackey}
Let $\CP$ be a filtered poset of length $\ell$ and $R$ be a unital commutative ring. Let $m\le \ell$ be a non-negative integer.
\begin{enumerate}
\item If $F\colon \CP\to R\Mod$ is a weak Mackey functor that has a quasi-unit in $\CP_{\le p}$ for every $p\in \CP$ with $d(p)\le m$, then $H_k(\CP;F)=0$ for all $k>\ell-m$.
\item If $G\colon \CP^\op\to R\Mod$ is a weak Mackey functor that has a quasi-unit in $\CP_{\le p}$ for every $p\in \CP$ with $d(p)\le m$, then $H^k(\CP;G)=0$ for all $k>\ell-m$.
\end{enumerate}
\end{thm}
\begin{proof}
We prove (1); the proof of (2) is analogous. For every $p\in \CP$ with $d(p)\le m$, the restriction $F\vert_{\CP_{\le p}}$ is a weak Mackey functor with a quasi-unit in $\CP_{\le p}$. By \cite[Lemma~3.4]{CD23}, $F\vert_{\CP_{\le p}}$ is cofibrant, which gives $H_k(\CP_{<q};F)=0$ for $k\ge 1$ and every $q\in \CP$ with $d(q)\le m+1$. The conclusion then follows from \cref{Thm-inductive2}.
\end{proof}

\begin{rmk}
	\label{rmk:weak-mackey-local}
Notice that this result is a refinement of \cite[Theorem 5.3]{CD-shell}: the poset here is not required to be CL-shellable. Moreover,
the proof of \cref{thm:VB_inductive_weak_mackey} applies \cite[Lemma~3.4]{CD23} independently to each restriction $F\vert_{\CP_{\le p}}$ for $d(p)\le m$, using no global coherence between different sub-posets. Consequently, the hypothesis can be weakened: it suffices to assume that, for every $p\in\CP$ with $d(p)\le m$, the restriction $F\vert_{\CP_{\le p}}$ admits a weak Mackey structure with quasi-unit, without requiring any kind of compatibility between the different weak Mackey structures for different $p$'s.
\end{rmk}
\subsection{Application to Knot Theory}
As an application, we explicitly show how \cref{Thm-inductive1} can be used to provide vanishing bounds for the unnormalised Khovanov homology $\overline{KH}^*(L)$ of a link $L$.

\begin{defn}
A \emph{link} is an embedding $L\colon \bigsqcup_{i=1}^r S^1 \hookrightarrow S^3$ of a disjoint union of circles into the 3-sphere, considered up to ambient isotopy. A link with a single component is called a \emph{knot}. We identify the link with its image in $S^3$.

Let $L\subset S^3$ be a link. A \emph{planar diagram} $D$ with $n$ crossings of $L$ is a generic immersion of the underlying circles in $\RR^2$, with finitely many transverse double points (the \emph{crossings}), decorated with over/under-crossing information. 
\end{defn}

\begin{figure}[h!]
\centering
\begin{tikzpicture}[scale=0.7, line width=1.2pt]

\begin{scope}[xshift=-3.5cm]
\draw[domain=0:360, samples=120, smooth, line width=1.2pt] 
plot ({sin(\x)+2*sin(2*\x)}, {cos(\x)-2*cos(2*\x)+0.5});
\node[font=\small] at (0, -3) {Knot (immersion)};
\end{scope}

\begin{scope}[xshift=3.5cm]
\tikzset{strand/.style={preaction={draw, white, line width=3pt}, line width=1pt, smooth, samples=50}}
\draw[domain=60:120, strand] plot ({sin(\x)+2*sin(2*\x)}, {cos(\x)-2*cos(2*\x)+0.5});
\draw[domain=180:300, strand] plot ({sin(\x)+2*sin(2*\x)}, {cos(\x)-2*cos(2*\x)+0.5});
\draw[domain=-60:60, strand] plot ({sin(\x)+2*sin(2*\x)}, {cos(\x)-2*cos(2*\x)+0.5});
\draw[domain=120:180, strand] plot ({sin(\x)+2*sin(2*\x)}, {cos(\x)-2*cos(2*\x)+0.5});

\node[font=\small] at (0,    2.5) {$1$};
\node[font=\small] at (-1.7, -0.4) {$2$};
\node[font=\small] at ( 1.7, -0.4) {$3$};

\node[font=\small] at (0, -3) {Planar Diagram $D$};
\end{scope}
\end{tikzpicture}
\caption{A generic immersion of the trefoil knot (left) and its planar diagram $D$ decorated with over/under-crossings (right).}
\label{fig:trefoil}
\end{figure}

Given a crossing of $D$, a \emph{smoothing} of $D$ is a diagram obtained by locally replacing the crossing with a pair of non-crossing arcs: the \emph{$0$-smoothing} connects the strands horizontally and the \emph{$1$-smoothing} connects them vertically, with the orientation described below in \cref{fig:smoothings}.

\begin{figure}[h!]
\centering
\begin{tikzpicture}[scale=0.7, line width=1.2pt]
\begin{scope}[xshift=0cm]
\draw (-0.8, 0.8) -- ( 0.8,-0.8);
\draw (-0.8,-0.8) -- (-0.15, -0.15);
\draw ( 0.15, 0.15) -- ( 0.8, 0.8);
\node at (0,-1.3) {crossing};
\end{scope}
\begin{scope}[xshift=4cm]
\draw (-0.8, 0.8) .. controls (-0.2, 0.2) and (-0.2,-0.2) .. (-0.8,-0.8);
\draw ( 0.8, 0.8) .. controls ( 0.2, 0.2) and ( 0.2,-0.2) .. ( 0.8,-0.8);
\node at (0,-1.3) {$1$-smoothing};
\end{scope}
\begin{scope}[xshift=-4cm]
\draw (-0.8, 0.8) .. controls (-0.2, 0.2) and ( 0.2, 0.2) .. ( 0.8, 0.8);
\draw (-0.8,-0.8) .. controls (-0.2,-0.2) and ( 0.2,-0.2) .. ( 0.8,-0.8);
\node at (0,-1.3) {$0$-smoothing};
\end{scope}
\draw[->, line width=0.8pt] (1.1,0) -- (2.7,0) node[midway, above] {\small $1$};
\begin{scope}[xshift=-4cm]
\draw[<-, line width=0.8pt] (1.1,0) -- (2.7,0) node[midway, above] {\small $0$};
\end{scope}
\end{tikzpicture}
\caption{}
\label{fig:smoothings}
\end{figure}

A \emph{complete smoothing} of $D$ is a choice of a smoothing at each crossing. The set of complete smoothings is indexed by $\{0,1\}^n$, and each complete smoothing $v\in\{0,1\}^n$ yields a planar diagram $D_v$ consisting of a disjoint union of circles. We denote by $|D_v|$ the number of circles in $D_v$.

The Boolean poset $\CP= \{0,1\}^n$ is a filtered poset of length $n$ with the filtration ${d(v) = \sum_{i=1}^n v_i}$. The \emph{Khovanov cube} is the contravariant functor
\[
F_{KH}\colon \CP^\op \to \Ab
\]
constructed as follows. Let $V=\ZZ[x]/(x^2)$ be the free $\ZZ$-module (actually a Frobenius algebra). To each vertex $v\in\CP$ one assigns the abelian group
\[
F_{KH}(v) = V^{\otimes |D_v|}.
\]
For a cover relation $w\prec v$ in $\CP^\op$, $v$ is obtained from $w$ by changing a single $0$ to a $1$ at position $k$. Thus, the morphism
\[
F_{KH}(w\prec v)\colon F_{KH}(v)\to F_{KH}(w)
\]
is induced by:
\begin{itemize}
\item the \emph{multiplication} $m\colon V\otimes V\to V$, $m(1\otimes x)=m(x\otimes 1)=x$; ${m(1\otimes 1)=1}$; $m(x\otimes x)=0$, when the $k$-th smoothing merges two circles of $D_v$ into one circle of $D_w$,
\item the \emph{comultiplication} $\Delta\colon V\to V\otimes V$, $\Delta(1)=1\otimes x+x\otimes 1$; $\Delta(x)=x\otimes x$, when the $k$-th smoothing splits one circle of $D_v$ into two circles of $D_w$.
\end{itemize}
This makes $F_{KH}$ a well-defined functor; see \cite[Section 1]{EverittKHov}.

\begin{rmk}
Traditionally, the Khovanov cube incorporates signs into the edge maps to ensure anti-commutativity, yielding a cochain complex that computes Khovanov homology, see \cite{MR1740682}. Moreover, the Frobenius algebra $V$ typically carries an internal quantum grading $q$. However, our framework relies on computing higher limits over posets, which requires the Khovanov cube to be strictly commutative. Furthermore, we disregard the $q$-grading and treat Khovanov homology as a singly-graded theory.
\end{rmk}

Denote by $\widehat{\CP}$ the poset obtained by adding an element $\hat{1}$ that is greater than every other element of $\CP$ except the maximum $\mathbf{1}=(1,\dots,1)$. We also denote by the same letter, ${F_{KH}}$, the extension by zero of $F_{KH}$ to $\widehat{\CP}$. That is, we set ${F_{KH}}(\hat{1})=0$ and the unique morphism $F_{KH}(\hat 1)\to F_{KH}(p)$ for any $p\in \CP$, see \cref{figure:khovanov_cube_trefoil}. Everitt and Turner prove that the higher limits of ${F_{KH}}$ in $\widehat{\CP}$ compute the unnormalised Khovanov homology of $L$; see \cite[Theorem~1.4]{EverittKHov}:
\begin{equation}\label{eq:iso_ET-kh}
\overline{KH}^i(L)=H^i\!\left(\widehat{\CP};{F_{KH}}\right).
\end{equation}

\begin{figure}[h!]
\centering
\begin{tikzpicture}[>=stealth, scale=0.8, transform shape]
\tikzstyle{box} = [draw, rectangle, minimum width=3cm, minimum height=2.2cm, fill=white]
\newcommand{\khstate}[3]{%
\begin{tikzpicture}[scale=.9, baseline=-3pt, line cap=round, line join=round]
\draw (110:0.4) arc (110:190:0.4); \draw (110:1.1) arc (110:190:1.1);
\draw (230:0.4) arc (230:310:0.4); \draw (230:1.1) arc (230:310:1.1);
\draw (350:0.4) arc (-10:70:0.4); \draw (350:1.1) arc (-10:70:1.1);

\ifnum#1=0
\draw (110:0.4) .. controls (90:0.8) .. (70:0.4);
\draw (110:1.1) .. controls (90:0.7) .. (70:1.1);
\else
\draw (110:0.4) .. controls (100:0.75) .. (110:1.1);
\draw (70:0.4) .. controls (80:0.75) .. (70:1.1);
\fi
\ifnum#2=0
\draw (230:0.4) .. controls (210:0.8) .. (190:0.4);
\draw (230:1.1) .. controls (210:0.7) .. (190:1.1);
\else
\draw (230:0.4) .. controls (220:0.75) .. (230:1.1);
\draw (190:0.4) .. controls (200:0.75) .. (190:1.1);
\fi
\ifnum#3=0
\draw (350:0.4) .. controls (330:0.8) .. (310:0.4);
\draw (350:1.1) .. controls (330:0.7) .. (310:1.1);
\else
\draw (350:0.4) .. controls (340:0.75) .. (350:1.1);
\draw (310:0.4) .. controls (320:0.75) .. (310:1.1);
\fi
\end{tikzpicture}
}

\def\knode#1#2#3#4#5#6{
\node[box] (#1) at (#2,#3) {};
\node[anchor=north east, inner sep=2pt, font=\scriptsize] at (#1.north east) {$#5$};
\node[anchor=south east, inner sep=2pt, font=\scriptsize] at (#1.south east) {#6};
\node[anchor=center, xshift=-10pt] at (#1.center) {#4};
}

\node at (-1.5, 3.5) {
\begin{tikzpicture}[scale=0.4]
\tikzset{strand/.style={preaction={draw, white, line width=3pt}, line width=1pt, smooth, samples=50}}
\draw[domain=60:120, strand] plot ({sin(\x)+2*sin(2*\x)}, {cos(\x)-2*cos(2*\x)+0.5});
\draw[domain=180:300, strand] plot ({sin(\x)+2*sin(2*\x)}, {cos(\x)-2*cos(2*\x)+0.5});
\draw[domain=-60:60, strand] plot ({sin(\x)+2*sin(2*\x)}, {cos(\x)-2*cos(2*\x)+0.5});
\draw[domain=120:180, strand] plot ({sin(\x)+2*sin(2*\x)}, {cos(\x)-2*cos(2*\x)+0.5});
\node at (0, 2.7) {\Large 1};
\node at (-1.8, -0.5) {\Large 2};
\node at (1.8, -0.5) {\Large 3};
\end{tikzpicture}
};

\knode{N000}{0}{0}{\khstate{0}{0}{0}}{V^{\otimes 2}}{000}

\knode{N100}{4.5}{3}{\khstate{1}{0}{0}}{V^{\phantom{\otimes 2}}}{100}
\knode{N010}{4.5}{0}{\khstate{0}{1}{0}}{V^{\phantom{\otimes 2}}}{010}
\knode{N001}{4.5}{-3}{\khstate{0}{0}{1}}{V^{\phantom{\otimes 2}}}{001}

\knode{N110}{9}{3}{\khstate{1}{1}{0}}{V^{\otimes 2}}{110}
\knode{N101}{9}{0}{\khstate{1}{0}{1}}{V^{\otimes 2}}{101}
\knode{N011}{9}{-3}{\khstate{0}{1}{1}}{V^{\otimes 2}}{011}

\knode{N111}{13.5}{-1.5}{\khstate{1}{1}{1}}{V^{\otimes 3}}{111}

\knode{N111hat}{13.5}{1.5}{}{0}{$\hat{1}$}

\tikzstyle{arr} = [->, thick]

\draw[<-, thick] (N000) -- (N100);
\draw[<-, thick] (N000) -- (N010);
\draw[<-, thick] (N000) -- (N001);

\draw[<-, thick] (N100) -- (N110);
\draw[<-, thick] (N100) -- (N101);
\draw[<-, thick] (N010) -- (N110);
\draw[<-, thick] (N010) -- (N011);
\draw[<-, thick] (N001) -- (N101);
\draw[<-, thick] (N001) -- (N011);

\draw[<-, thick] (N110) -- (N111);
\draw[<-, thick] (N101) -- (N111);
\draw[<-, thick] (N011) -- (N111);

\draw[<-, thick] (N110) -- (N111hat);
\draw[<-, thick] (N101) -- (N111hat);
\draw[<-, thick] (N011) -- (N111hat);

\end{tikzpicture}
\caption{The Khovanov cube for the trefoil $L$.}\label{figure:khovanov_cube_trefoil}
\end{figure}

However, they compute these higher limits using a cellular approach as shown in \cite{EverittCel}. Such a global combinatorial method obscures the inductive nature of the link smoothings. Our inductive framework naturally captures this local simplification process to produce vanishing bounds for the Khovanov homology.

\begin{rmk}
Consider an element $p\in \CP$. The lower subposet $\CP_{\le p}$ consists of all states where the crossings associated with the $0$-coordinates of $p$ are permanently fixed to their $0$-smoothing. In other words, $\CP_{\le p}$ is naturally isomorphic to the Boolean poset of a simplified diagram $D'$, obtained from $D$ by $0$-smoothing all those fixed crossings. And, by definition of the Khovanov cube, the restriction of the functor $F_{KH}\vert_{\CP_{\le p}}$ coincides with the Khovanov cube of $D'$. This formal self-similarity is the underlying mechanism that enables our inductive bounds.
\end{rmk}

To illustrate \cref{Thm-inductive1} applied to Khovanov homology, we first need the following lemmas.

\begin{lem}
\label{lem:KH-auxiliar-PB}
Let $\CP$ be the Boolean poset $\{0,1\}^n$ and $F\colon \widehat{\CP}^\op\to \Ch(R)$ be a functor. Then, the limit $\lim_{\widehat{\CP}} F$ coincides with the pullback of the following diagram:
\begin{equation}\label{eq:pull-khov}
\begin{tikzcd}
&  F(\hat{1})\arrow[d] \\
F(\mathbf{1})\arrow[r] & \displaystyle\slim_{\CP_{<\mathbf{1}}} F.
\end{tikzcd}
\end{equation}
\end{lem}
\begin{proof}
Let us denote by $C$ the pullback of \eqref{eq:pull-khov}. A direct computation shows that the morphism $\lim_{\widehat{\CP}} F\to C$ defined by 
\[
(f_i)_{i\in \widehat{\CP}}\mapsto (f_{\mathbf{1}}, f_{\widehat{1}},\operatorname{Res}^{\widehat{\CP}}_{\CP_{<\mathbf{1}}}((f_i)_{i\in \widehat{\CP}}))
\]
is an isomorphism, where $\operatorname{Res}^{\widehat{\CP}}_{\CP_{<\mathbf{1}}}$  denotes the restriction morphism. 
\end{proof}
\begin{lem}
\label{lem:vanishing-bound-khovanov}
Let $D$ be a diagram of a link $L$ and $F_{KH}\colon \widehat{\CP}^\op\to \Ch(R)$ be the Khovanov cube. Let $N$ be the maximum of the heights of $\NR F_{KH}$ in $\CP_{< \mathbf{1}}$. If the $(N+1)$-th Khovanov homology group of $L$ vanishes, i.e.\ $\overline{KH}^{N+1}(L)=0$, then 
\[
H^{N}(\CP_{<\mathbf{1}};F_{KH})=0.
\]
\end{lem}
\begin{proof}
Let $\NR' F_{KH}\colon \CP^\op\to \Ch(\Ab)$ be the fibrant replacement of $F_{KH}$ as in Definition~\ref{def_fibrant_rep} in $\CP_{<\mathbf{1}}$, and defined in $\mathbf{1}$ and $\hat 1$ as the mapping cocylinder of $\varepsilon_{\mathbf{1}}$ and $\varepsilon_{\hat 1}$ respectively, see Remark~\ref{rmk:fib-choice-cocyl}. 

By \eqref{eq:iso_ET-kh} and \cref{hig_lim_fibrant_rep}(1), we have:
\[
\overline{KH}^{N+1}(L)=H^{N+1}({\lim}_{\widehat{\CP}} \NR' {F_{KH}}).
\]
Next, by Lemma~\ref{lem:KH-auxiliar-PB}, we have that $\lim_{\widehat{\CP}} \NR' {F_{KH}}$ is the pullback of the diagram \eqref{eq:pull-khov}, and by hypothesis, $\alt(\lim_{\CP_{<\mathbf{1}}}\NR' F_{KH})=N$, so  we have
\[
\begin{tikzcd}[row sep= small]
\displaystyle\slim_{\widehat{\CP}} {\NR' F_{KH}}^{N}\arrow[r, "\partial"] \arrow[d, Rightarrow,no head]                                                                                                & \displaystyle\slim_{\widehat{\CP}} \NR' F_{KH}^{N+1}  \arrow[d, Rightarrow,no head]                          \\
\displaystyle\slim_{\CP_{<\mathbf{1}}}{\NR' F_{KH}}^{N-1}\oplus \slim_{\CP_{<\mathbf{1}}}{\NR' F_{KH}}^{N}\oplus \slim_{\CP_{<\mathbf{1}}}{\NR' F_{KH}}^{N-1} \arrow[r, "\partial"] & \displaystyle\slim_{\CP_{<\mathbf{1}}}\NR' F_{KH}^{N}\oplus \slim_{\CP_{<\mathbf{1}}}{\NR' F_{KH}}^{N} \\
{(x_1,x_2,x_3)} \arrow[r, maps to]                                                                                                      & {(x_2-\overline\partial x_1,x_2-\overline\partial x_3)}                               
\end{tikzcd}
\]
where $\overline\partial$ is the corresponding differential of the cochain complex $\lim_{\CP_{<1}}{\NR' F_{KH}}$.
Since $\overline{KH}^{N+1}(L)=0$, the differential $\partial$ is surjective. This implies that $\overline\partial$ is surjective, and thus 
\[
H^N(\CP_{<\mathbf{1}};F_{KH})=H^N\left(\lim_{\CP_{<\mathbf{1}}}\NR F_{KH}\right)=0.
\]
\end{proof}

Although the Khovanov homology of the knot $8_{19}$, also known as $T(3,4)$, is well-known, we show how Theorem~\ref{Thm-inductive1} produces a sharper bound than the classical one given by the crossing number.

\begin{figure}[h!]
\centering
\begin{tikzpicture}
\node[anchor=south west,inner sep=0] (image) at (0,0) {\includegraphics[width=0.35\textwidth]{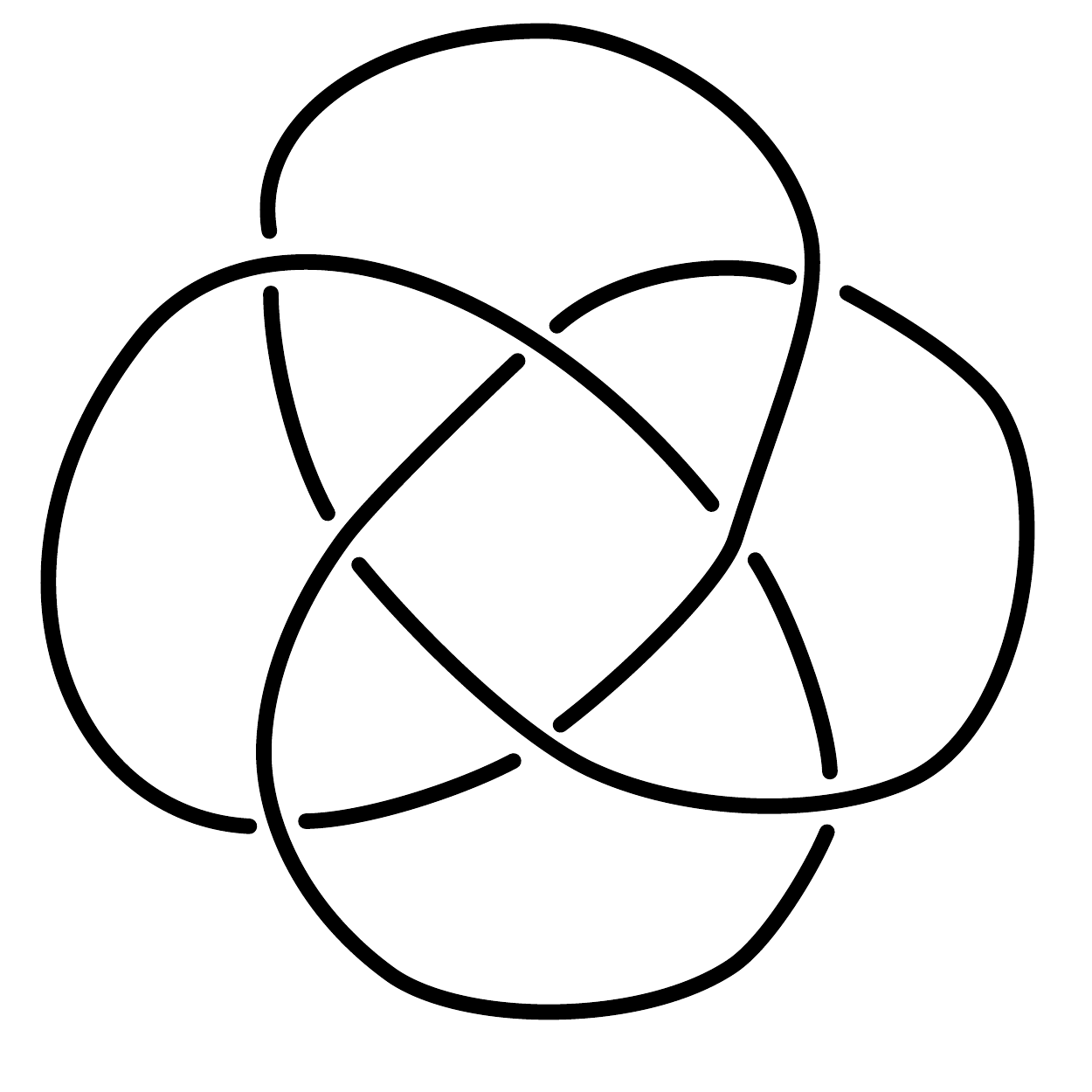}};
\begin{scope}[x={(image.south east)},y={(image.north west)}]
\node at (0.2, 0.8) {$1$};
\node at (0.8, 0.8) {$2$};
\node at (0.81, 0.2) {$3$};
\node at (0.2, 0.2) {$4$};
\node at (0.2, 0.5) {$5$};
\node at (0.48, 0.75) {$6$};
\node at (0.75, 0.5) {$7$};
\node at (0.49, 0.25) {$8$};
\end{scope}
\end{tikzpicture}
\caption{Diagram of the torus knot $8_{19}$, also known as $T(3,4)$.}
\label{fig:knot8_19}
\end{figure}
Let $D$ be the diagram of $8_{19}$ as shown in Figure~\ref{fig:knot8_19}, and let $F_{KH}\colon\CP^\op\to \Ab$ be the Khovanov cube of $D$.
We prove by induction that $H^k(\CP_{<p};F_{KH})=0$ for all $k\ge 6$ and for all $p\in \widehat\CP$. 

For $p\in \CP$ with $d(p)\le 6$, trivially, $H^k(\CP_{<p};F_{KH})=0$ for all $k\ge 6$. So we need to prove it for $d(p)=7$ and $8$.

\begin{figure}[h!]
\centering
\begin{tikzpicture}
\node[anchor=south west,inner sep=0] (image) at (0,0) {\includegraphics[width=0.35\textwidth]{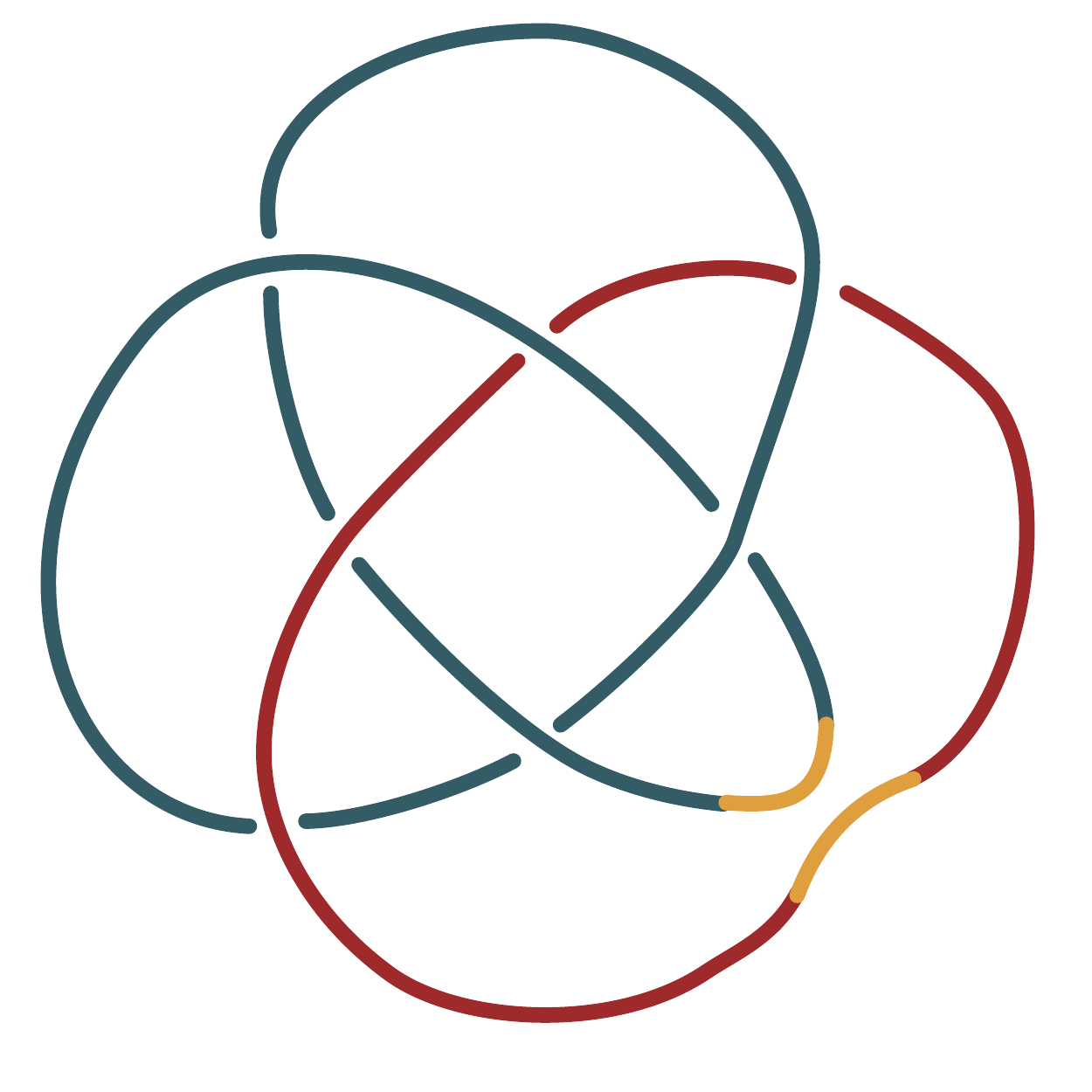}};
\begin{scope}[x={(image.south east)},y={(image.north west)}]
\node at (0.2, 0.8) {$1$};
\node at (0.8, 0.8) {$2$};
\node at (0.82, 0.18) {$3$};
\node at (0.2, 0.2) {$4$};
\node at (0.2, 0.5) {$5$};
\node at (0.48, 0.75) {$6$};
\node at (0.75, 0.5) {$7$};
\node at (0.49, 0.25) {$8$};
\end{scope}
\end{tikzpicture}
\begin{tikzpicture}
\node[anchor=south west,inner sep=0] (image) at (0,0) {\includegraphics[width=0.35\textwidth]{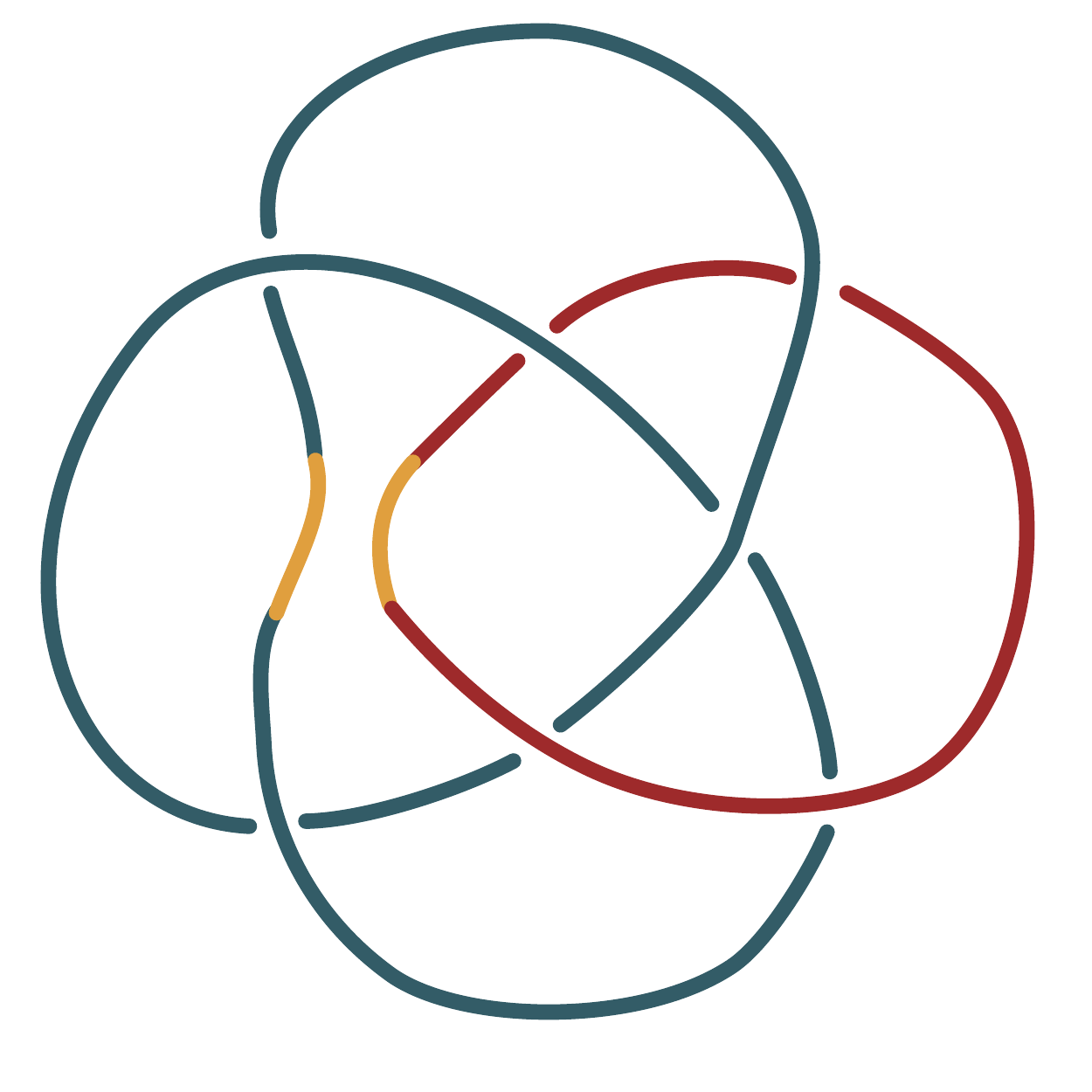}};
\begin{scope}[x={(image.south east)},y={(image.north west)}]
\node at (0.2, 0.8) {$1$};
\node at (0.8, 0.8) {$2$};
\node at (0.82, 0.18) {$3$};
\node at (0.2, 0.2) {$4$};
\node at (0.2, 0.5) {$5$};
\node at (0.48, 0.75) {$6$};
\node at (0.75, 0.5) {$7$};
\node at (0.49, 0.25) {$8$};
\end{scope}
\end{tikzpicture}
\caption{Two examples of $0$-smoothings of a crossing of $D$.}\label{fig_0-smoothings}
\end{figure}
The symmetry of the diagram $D$ implies that every $0$-smoothing of a crossing of $D$ yields an equivalent diagram, see Figure~\ref{fig_0-smoothings}. Thus, different choices of $p$ with ${d(p)=7}$ yield equivalent Khovanov cubes, $F_{KH}\vert_{\CP_{\le p}}$. Fix $p\in \CP$ with $d(p)=7$, to prove $\alt(\NR F_{KH}(p))\le 5$, it is enough to prove that $\alt(\NR F_{KH}(q))\le 5$ for all $q\le p$ with $d(q)=6$. By construction, $\alt(\NR F_{KH}\vert_{\CP_{\le q}})\le 5$; if the strict inequality holds the result is obvious, so assume that $\alt(\NR F_{KH}\vert_{\CP_{\le q}})=5$. Then, $F_{KH}\vert_{\CP_{\le q}}$ is the Khovanov cube of either $L6n1$ or the mirror image of $5_1$, see Figure~\ref{fig:00-smoothings}. In any case, we have:
\[
\overline{KH}^k(L6n1)=\overline{KH}^k(5_1)=0 \text{ for }k>5.
\]
By Lemma~\ref{lem:vanishing-bound-khovanov}, we have $H^5(\CP_{<q};\NR F_{KH})=0$. We apply Lemma~\ref{prop_VB_inductive equivalences} to $F_{KH}\vert_{\CP_{\le q}}$, which implies that $\alt(\NR F_{KH}(q))=5$. 

\begin{figure}[h!]
\centering
\begin{tikzpicture}
\node[anchor=south west,inner sep=0] (image) at (0,0) {\includegraphics[width=0.75\textwidth]{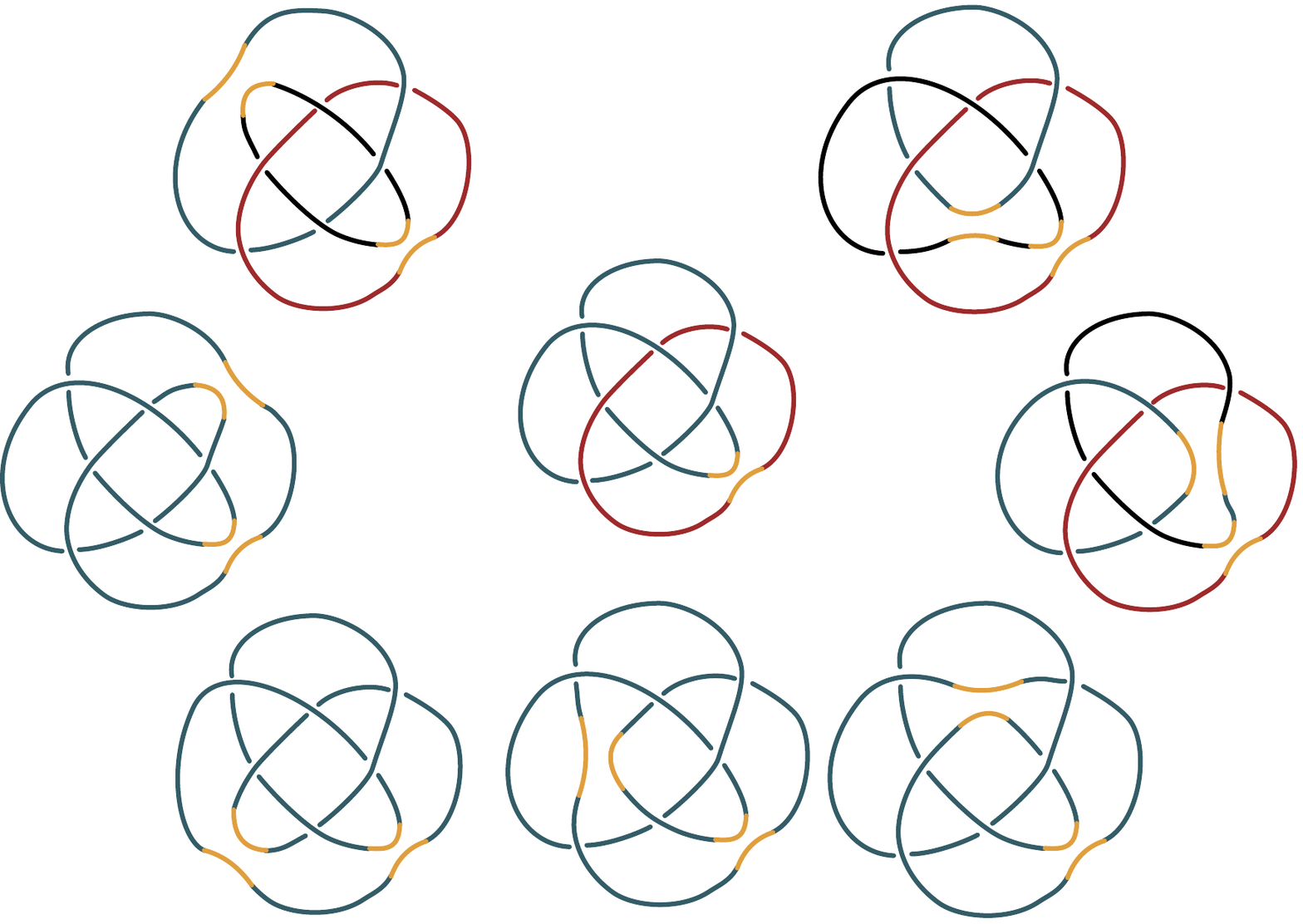}};
\begin{scope}[x={(image.south east)},y={(image.north west)}]

\coordinate (center) at ([shift={(0,0.08)}]image.center);
\coordinate (a) at ([shift={(-0.15,0)}]image.north east);
\coordinate (b) at ([shift={(0,-0.05)}]image.east);
\coordinate (c) at ([shift={(-0.1,0)}]image.south east);
\coordinate (d) at (image.south);
\coordinate (e) at ([shift={(0.1,0)}]image.south west);
\coordinate (f) at ([shift={(0,-0.05)}]image.west);
\coordinate (g) at ([shift={(0.15,0)}]image.north west);

\draw[ultra thick, ->, shorten <=15mm, shorten >=35mm] (center) -- (a);
\draw[ultra thick, ->, shorten <=15mm, shorten >=30mm] (center) -- (b);
\draw[ultra thick, ->, shorten <=15mm, shorten >=37mm] (center) -- (c);
\draw[ultra thick, ->, shorten <=15mm, shorten >=31mm] (center) -- (d);
\draw[ultra thick, ->, shorten <=15mm, shorten >=38mm] (center) -- (e);
\draw[ultra thick, ->, shorten <=15mm, shorten >=30mm] (center) -- (f);
\draw[ultra thick, ->, shorten <=15mm, shorten >=35mm] (center) -- (g);
\end{scope}
\end{tikzpicture}
\caption{All $0$-smoothings of $D'$.}
\label{fig:00-smoothings}
\end{figure}

Finally, for $d(p)=8$, we have to analyze the objects $q<p$ of degree $d(q)=7$. By the same symmetry argument above, the Khovanov cube restricted to $\CP_{\le q}$ corresponds to the link $L7n1$, and $\overline{KH}^6(L7n1)=0$. By Lemma~\ref{lem:vanishing-bound-khovanov}, we have $H^5(\CP_{<p};\NR F_{KH})=0$, and this implies that $\alt(\NR F_{KH}(p))= 5$. Thus, we conclude that for every $p\in \widehat\CP$,
\[
H^k(\CP_{<p};F_{KH})=0\text{ for }k\ge6.
\]
By Theorem~\ref{Thm-inductive1}, 
\[
\overline{KH}^k(L;\ZZ)\cong H^k(\widehat\CP;F_{KH})=0\text{ for }k>6.
\]

\begin{rmk}
Our inductive argument successfully shows the vanishing of the Khovanov homology of $8_{19}$ for all $k > 6$, which already strictly improves upon the classical crossing-number threshold. Although it is known that $\overline{KH}^6(8_{19})=0$, lowering our bound to capture this final degree would require exhaustive computations of the fibrant replacements. Furthermore, notice that $\CP$ is CL-shellable, and for $8_{19}$ the Khovanov cube $F_{KH}$ satisfies the stability property defined in \cite{CD-shell} in certain degrees. Therefore, these two techniques could be combined to develop new computational tools for Khovanov homology. The author intends to explore this further in future work.
\end{rmk}

\newcommand{\etalchar}[1]{$^{#1}$}

\end{document}